\documentclass{amsart}

\usepackage{amssymb}

\usepackage{amsfonts}
\usepackage{mathrsfs}
\usepackage{bbm}

\theoremstyle{plain}
    \newtheorem{theorem}{Theorem}[section]
    \newtheorem*{theorem*}{Theorem}
    \newtheorem{lemma}[theorem]{Lemma}
    \newtheorem{proposition}[theorem]{Proposition}
    \newtheorem{corollary}[theorem]{Corollary}

\theoremstyle{definition}
    
    \newtheorem{remark}{Remark}[section]
    
\numberwithin{equation}{section}

\renewcommand{\r}{\right}
\newcommand{\eps}{\varepsilon}
\newcommand{\odd}{{\mathord{\rm odd}}}
\newcommand{\1}{\mathbbm{1}}
\newcommand{\C}{\mathbb{C}}
\newcommand{\N}{\mathbb{N}}
\newcommand{\R}{\mathbb{R}}

\DeclareMathOperator{\im}{Im}
\DeclareMathOperator{\re}{Re}

\DeclareMathOperator{\Span}{span}

\begin{document}

\title[Threshold odd solutions to NLS in 1d]{Threshold odd solutions to the nonlinear Schr\"{o}dinger equation in one dimension}
\author[S. Gustafson]{Stephen Gustafson}
\address[S. Gustafson]{University of British Columbia, 1984 Mathematics Rd., Vancouver, Canada V6T1Z2.}
\email{gustaf@math.ubc.ca}
\author[T. Inui]{Takahisa Inui}
\address[T. Inui]{Department of Mathematics, Graduate School of Science, Osaka University, Toyonaka, Osaka, Japan 560-0043.
\newline 
University of British Columbia, 1984 Mathematics Rd., Vancouver, Canada V6T1Z2.}
\email{inui@math.sci.osaka-u.ac.jp 
}
\date{\today}
\date{\today}
\maketitle

\begin{abstract}
We consider odd solutions to the Schr\"{o}dinger equation with the $L^2$-supercritical power type nonlinearity in one dimensional Euclidean space. It is known that the odd solution scatters or blows up if its action is less than twice as that of the ground state. In the present paper, we show that the odd solutions with the action as twice as that of the ground state scatter or blow up. 
\end{abstract}

\setcounter{tocdepth}{1}
\tableofcontents


\section{Introduction}

\subsection{Background}

We consider the following nonlinear Schr\"{o}dinger equation in one dimensional Euclidean space:
\begin{align}
\label{NLS}
\tag{NLS}
	\begin{cases}
	i\partial_t u+ \partial_x^2 u+|u|^{p-1}u=0, & (t,x) \in I \times \mathbb{R},
	\\
	u(0,x)=u_0(x), & x \in \mathbb{R}, 
	\end{cases}
\end{align}
where $I$ denotes a time interval and $p>5$. 
It is known that the Cauchy problem of \eqref{NLS} is locally well-posed in the energy space $H^1(\mathbb{R})$. That is, there exist $T>0$ and a unique solution $u \in C((-T,T):H^1(\mathbb{R}))$ if $u_0 \in H^1(\mathbb{R})$. Moreover, the energy and mass are conserved by the flow, where the energy $E$ and mass $M$ are defined by 
\begin{align*}
	E(u):= \frac{1}{2} \|\partial_x u\|_{L^2}^2 - \frac{1}{p+1} \|u\|_{L^{p+1}}^{p+1}, \quad
	M(u):=  \| u\|_{L^2}^2.
\end{align*}
We also have blow-up alternative. Namely, the $\dot{H}^{1}$-norm of the solution diverges at the maximal existence time if the time is finite. See \cite{GiVe79} and \cite{Caz03,LiPo15} for these results. 
In this paper, we are interested in the global behavior of the solution. 

Since the work by Kenig and Merle \cite{KeMe06}, many researchers have investigated the global dynamics of the nonlinear Schr\"{o}dinger equations below the ground state. See \cite{HoRo08,DHR08,HoRo10} for 3d cubic NLS. For the one dimensional NLS, i.e. \eqref{NLS}, Akahori and Nawa \cite{AkNa13}, where they in fact treated the equations with $L^2$-super and energy-subcritical nonlinearity in all dimensions, determined the global behavior of the solutions below the ground state.
The ground state is the solution with the formula $u(t,x)=e^{it}Q(x)$, where $Q(x)$ is the radial positive solution of the elliptic equation 
\begin{align}
\label{elliptic}
	-\partial_x^2 Q + Q -Q^{p}=0
\end{align} 
and it is explicitly given by
\begin{align}
\label{exp}
	Q(x)= \left( \frac{p+1}{2}\right)^{\frac{1}{p-1}}  \left\{\cosh \left(\frac{p-1}{2} x \right) \right\}^{-\frac{2}{p-1}}
\end{align}
in the one dimensional case. 
They showed that if the initial data satisfies the energy and mass condition $E(u_0)M(u_0)^{\sigma} < E(Q)M(Q)^\sigma$, where $\sigma=(p-3)/(p-5)$, then the behavior of the solution is determined by the sign of the virial functional $K$ at initial time, where the functional is defined by
\begin{align*}
	K(u):= \|\partial_x u \|_{L^2}^2-\frac{p-1}{2(p+1)}\| u \|_{L^{p+1}}^{p+1}.
\end{align*}
Precisely, they showed the following: if $K(u_0)\geq 0$, then the solution is global and scatters in both time directions, i.e. there exist $u_{\pm} \in H^1(\mathbb{R})$ such that 
\begin{align*}
	\| u(t) -e^{it\partial_x^2}u_{\pm}\|_{H^1} \to 0 \text{ as } t \to \pm \infty.
\end{align*}
If $K(u_0)< 0$, then the solution blows up in finite time or grows up at infinite time. In addition, if the variance of the solution is finite, i.e. $\|xu_0\|_{L^2} < \infty$, then the solution blows up in finite time in both time directions. (See also Fang, Xie, and Cazenave \cite{FXC11}  for the scattering result.)

The global behavior of the solutions with $E(u_0)M(u_0)^{\sigma} = E(Q)M(Q)^\sigma$, which are called threshold solutions, was also investigated by Duyckaerts and Roudenko \cite{DuRo10} for 3d cubic NLS. Recently, for one dimensional NLS, Campos, Farah, and Roudenko \cite{CFR20} studied them, where they treated $L^2$-supercritical NLS, including the energy-critical one, in general dimensions (see also \cite{DuMe09}). They showed the following theorems in the one dimensional case:

\begin{theorem*}[\cite{CFR20}]
There exist radial (even) solutions $Q^+$ and $Q^-$ to \eqref{NLS} on at least $[0,\infty)$ such that 
\begin{itemize}
\item $M(Q^{\pm})=M(Q)$ and $E(Q^{\pm})=E(Q)$,
\item $Q^{\pm}$ satisfies 
\begin{align*}
	\|Q^{\pm}(t) - e^{it}Q\|_{H^1} \leq Ce^{-ct} \text{ for all } t>0
\end{align*}
for some constants $C,c>0$,
\item $K(Q^{+}(0))<0$ and $Q^{+}$ blows up in finite negative time,
\item $K(Q^{-}(0))> 0$ and $Q^{-}$ scatters backward in time.
\end{itemize}
\end{theorem*}

\begin{theorem*}[\cite{CFR20}]
Let $u_0$ satisfy $E(u_0)M(u_0)^{\sigma} = E(Q)M(Q)^\sigma$ and $u$ be the solution to \eqref{NLS}. Then we have the following: 
\begin{enumerate}
\item If $K(u_0)>0$, then $u$ is global and either $u$ scatters in both time directions or $u=Q^{-}$ up to the symmetries of the equation. 
\item If $K(u_0)=0$, then $u=Q$ up to the symmetries of the equation. 
\item If $K(u_0)<0$ and $\|xu_0\|_{L^2}<\infty$, then either $u$ blows up in finite both time or $u=Q^{+}$ up to the symmetries of the equation. 
\end{enumerate}
\end{theorem*}

The special solutions $Q^{\pm}$, which converge to the ground state, appear at the threshold, though we do not have them below the ground state. 

Now, let us assume that $u_0$ is an odd function. Then the solution $u$ is also odd. In this case, 
the second author determined the global behavior of the solutions below the ground state as follows:

\begin{theorem*}[{\cite[Theorem 1.2]{Inui17}}]
Let $u_0$ be odd and $u$ be the solution to \eqref{NLS}. 
Assume that $u_0$ satisfies $E(u_0)M(u_0)^{\sigma} < 2^{1+\sigma} E(Q)M(Q)^\sigma$. Then we have the following. 
\begin{enumerate}
\item If $K(u_0)\geq 0$, then $u$ is global and scatters in both time directions. 
\item If $K(u_0)<0$, then either $u$ blows up in finite time or grows up at infinite time. 
\end{enumerate}
\end{theorem*}

\begin{remark}
In Theorem 1.2 in \cite{Inui17}, the second author gave the condition in terms of the action and the frequency $\omega$ instead of the energy and mass condition $E(u_0)M(u_0)^{\sigma} < 2^{1+\sigma} E(Q)M(Q)^\sigma$. As seen in Section \ref{sec2.1} below, the formulation in \cite{Inui17} 
is equivalent to the energy and mass condition. 
\end{remark}

Since the ground state $e^{it}Q$ and the special solutions $Q^{\pm}$ are even, the assumption of odd symmetry excludes them. That is why we can determine the global dynamics 
of odd solutions above the ground state. 

In the odd case, the energy and mass condition $E(u_0)M(u_0)^{\sigma} <2^{1+\sigma} E(Q)M(Q)^\sigma$ is related to a minimization problem restricted to odd functions 
just as the relation $E(u_0)M(u_0)^{\sigma} <E(Q)M(Q)^\sigma$ is. 
Roughly, the condition comes from a minimizing sequence of odd functions $\{\varphi_n\}_{n \in \mathbb{N}}=\{Q(\cdot -n) - Q(\cdot + n)\}_{n \in \mathbb{N}}$ with $E(\varphi_n)M(\varphi_n)^\sigma \to 2^{1+\sigma} E(Q)M(Q)^\sigma$ as $n \to \infty$. Moreover, the minimizer is not attained. (See Section \ref{sec2.1} below for more details.) Therefore, we expect that on the threshold $E(u_0)M(u_0)^{\sigma} =2^{1+\sigma} E(Q)M(Q)^\sigma$ 
there are no odd solutions like the ground state and the special solutions. 
%
%
In the present paper, we will show that odd solutions at the threshold scatter or blow up. 

\begin{remark}
Odd solutions for \eqref{NLS} can also be regarded as solutions of NLS on half line $[0,\infty)$ with Dirichlet zero boundary condition at the origin. Thus, the result in the paper holds for NLS on half line with this Dirichlet zero boundary. The same goes for NLS on the star graph, which is a metric graph connecting half lines, with Dirichlet zero boundary
condition. 
\end{remark}


\subsection{Main result}

We use $H_{\odd}^{1}(\mathbb{R})$ to denote the set of odd functions in $H^1(\mathbb{R})$. We obtain the following result: 

\begin{theorem}
\label{thm1.2}
Let $u_0\in H_{\odd}^{1}(\mathbb{R})$. Assume $E(u_0)M(u_0)^{\sigma} = 2^{1+\sigma}E(Q)M(Q)^\sigma$. Then we have the following:
\begin{enumerate}
\item If $K(u_0)>0$, then the solution scatters in both time directions. 
\item In addition, we assume $\|x u_0\|_{L^2} < \infty$. If $K(u_0)<0$, then the solution blows up in finite time in 
both time directions.  
\end{enumerate}
\end{theorem}

\begin{remark}
$K(u_0)=0$ is not possible under the assumption $u_0\in H_{\odd}^{1}(\mathbb{R})$ and $E(u_0)M(u_0)^{\sigma} = 2^{1+\sigma}E(Q)M(Q)^\sigma$
(see Section \ref{sec2.1}).
\end{remark}

\begin{remark}
The assumption of oddness is essential. Indeed, Nguy\~{\^{e}}n \cite{Ngu19} obtained a "two peak" solution $u$ such that 
\begin{align*}
\|u(t) - e^{i\gamma(t)}\sum_{k=1}^2 Q(\cdot - x_k(t))\|_{H^1} \lesssim t^{-1}	
\end{align*}
for all $t>0$, where $x_1=-x_2$ and $|x_1(t)-x_2(t)|= 2(1+o(1)) \log t$ as $t \to \infty$. This solution satisfies $E(u_0)M(u_0)^{\sigma} = 2^{1+\sigma}E(Q)M(Q)^\sigma$. We note that the solitons have the same sign, so $u$ is not odd.
\end{remark}

\subsection{Idea of proof}

Our proof of the scattering result is based on the contradiction argument by Duyckaerts, Landoulsi, and Roudenko \cite{DLR20}, where they consider 3d cubic NLS outside an obstacle (see also \cite{MMZ21} and \cite{ArIn21}). If the statement fails, combining a modulation argument and a concentration compactness argument, we construct a critical element at the threshold which has the following compactness property: There exists $x(t)$ such that $u(t,\cdot) - \{\psi(t, \cdot-x(t))-\psi(t, - \cdot -x(t))\}$ converges to $0$ for some $\psi$. If $x(t)$ goes to infinity, then the assumption on the virial functional gives us a contradiction, and thus $x(t)$ must be bounded. However, the boundedness of $x(t)$ implies a contradiction to a modulation argument.  

For the blow-up result, we apply the argument by Duyckaerts and Roudenko \cite{DuRo10} (see also \cite{CFR20}). If the solution is global, the finite variance and the negativity of the virial functional imply that the modulation parameter $y$ is bounded. However, this gives a contradiction. 

The main difficulty appears in the modulation argument. For NLS with potential as in \cite{MMZ21,ArIn21}, we consider the linearization around the ground state $Q$ and thus we can control the translation parameter $y$ easily. On the other hand, in our setting, we need to consider the linearization around $Q(\cdot -y) -Q(\cdot+y)$. Therefore, it is more difficult to control the translation parameter $y$. This is similar to \cite{DLR20}. They overcame this difficulty by using the result for NLS on $\mathbb{R}^3$. In our proof, the following plays a very important role: 
\begin{align*}
	S(Q(\cdot -y) -Q(\cdot+y)) -2S(Q) \approx e^{-2y},
\end{align*}
where $S=E+M$ (see Lemma \ref{key}). Since this is positive, we can control the translation parameter $y$. (Since $S(Q(\cdot -y)+Q(\cdot+y)) -2S(Q) \approx -e^{-2y}$, the oddness is very important.)
By using this, we can control the modulation parameters.

\subsection{Notation}



We define 
\begin{align*}
	\langle f , g \rangle :=\re \int_{\mathbb{R}} f(x)\overline{g(x)} dx.
\end{align*}

We use following notation: for $y\geq 0$, 
\begin{align*}
	\mathcal{T}_{y} f(x) &:=f(x-y),
	\\
	\mathcal{R}_{y} f(x) &:=f(x-y)-f(-x-y)=
	\mathcal{T}_{y} f(x) - (\mathcal{T}_{y} f) (-x).
\end{align*}
The function $\mathcal{R}_{y} f$ is an odd function even when $f$ is neither odd nor even. 
For an even function $f$, we have
\begin{align*}
	\mathcal{R}_{y} f(x) =\mathcal{T}_{y} f(x) - \mathcal{T}_{-y} f(x) =f(x-y)-f(x+y).
\end{align*}

We define a smooth and even cut-off function $\chi_R: \mathbb{R} \to [0,1]$ by
\begin{align*}
	\chi_{1}(x):= 
	\begin{cases}
	1 & (|x|>1)
	\\
	0 & (|x|<1/2)
	\end{cases},
	\quad 
	\chi_{R}(x):=\chi_1\left( \frac{x}{R}\right)
\end{align*}
and we use 
\begin{align*}
	\chi_{R}^c(x):=1-\chi_{R}(x), \quad 
	\chi_{R}^{+}(x):=\1_{(0,\infty)}(x)\chi_{R}(x), \quad 
	\chi_{R}^{-}(x):=\1_{(-\infty,0)}(x)\chi_{R}(x),
\end{align*}
where $\1_{A}$ denotes the characteristic function on a set $A$. 
We define
\begin{align*}
	\mathcal{G}_{R,y}f(x):=\chi_{R}^{+}(x)\mathcal{T}_{y}f(x) 
	- \chi_{R}^{-}(x)\mathcal{T}_{-y}f(x).
\end{align*}
We set 
\begin{align*}
	\mu_{\omega} (u) &:= 2\|\partial_x Q_{\omega}\|_{L^2}^2 - \|\partial_x u\|_{L^2}^2,
\end{align*}
where $Q_\omega(x) = \omega^{\frac{1}{p-1}} Q (\sqrt{\omega} x)$,
and we denote $\mu=\mu_1$ for simplicity. 

We use the notation $A\lesssim B$ if there exists a positive constant $C$ such that $A \leq C B$. The notation $A\approx B$ means that both $A \lesssim B$ and $A\gtrsim B$ hold.


\section{Preliminaries}

\subsection{Variational argument}
\label{sec2.1}

In this section, we consider the variational structure for odd functions.

\subsubsection{Minimizing problem for odd functions}
Let $\omega>0$. 
We define the action by
\begin{align*}
	S_{\omega}(f)
	&:=E(f)+\frac{\omega}{2}M(f)
	=\frac{1}{2}\|\partial_x f \|_{L^2}^2
	+\frac{\omega}{2}\|f \|_{L^2}^2
	-\frac{1}{p+1}\|f \|_{L^{p+1}}^{p+1}.
\end{align*}
We consider the following minimizing problems:
\begin{align*}
	l_{\omega}&:=\inf \{S_{\omega}(f): f \in H^1(\mathbb{R})\setminus\{0\}, K(f)=0\},
	\\
	l_{\omega}^{{\rm odd}}&:=\inf \{S_{\omega}(f): f \in H_{{\rm odd}}^1(\mathbb{R})\setminus\{0\}, K(f)=0\}.
\end{align*}
It is known that $l_{\omega}$ is attained by the ground state $Q_{\omega}$, i.e. $l_{\omega}=S_{\omega}(Q_{\omega})$, where
we recall $Q_{\omega}(x)=\omega^{\frac{1}{p-1}} Q(\sqrt{\omega} x)$, and $Q_{\omega}$ is the radial positive solution to the elliptic equation $-\partial_x^2 Q_{\omega} + \omega Q_{\omega} -Q_{\omega}^{p}=0$. We note $Q=Q_1$ and set  $S:=S_{1}$. 

We prove that $l_{\omega}^{\odd}=2l_{\omega}$ and there is no minimizer. 

\begin{lemma}
We have $l_{\omega}^{{\rm odd}}=2 l_{\omega}(=2S_{\omega}(Q_{\omega}))$. 
\end{lemma}

\begin{proof}
Define a sequence $\{\varphi_n\}$ of odd functions by $\varphi_n(x) := \mathcal{R}_{n}Q_{\omega}(x)$. Then $\{\varphi_n\}$ satisfies $S_{\omega}(\varphi_n) \to 2 S_{\omega}(Q_{\omega})$ and $K(\varphi_n)\to 0$. 
Take a parameter $\lambda_n$ such that $K(\lambda_n \varphi_n)=0$. Then $\lambda_n \to 1$ and thus $ S_{\omega}(\lambda_n \varphi_n) \to 2S(Q_{\omega})$. Therefore, we get $l_{\omega}^{{\rm odd}} \leq 2 S_{\omega}(Q_{\omega})=2l_{\omega}$. 

On the other hand, taking a minimizing sequence $\{\varphi_n\}$ of $l^{\rm odd}$
We note that $\1_{(0,\infty)}\varphi_n \in H^1(\mathbb{R})$ and $\partial_x(\1_{(0,\infty)}\varphi_n)=\1_{(0,\infty)}\partial_x \varphi_n$. Then, it holds that
\begin{align*}
	&2S(\1_{(0,\infty)}\varphi_n)=S(\1_{(0,\infty)}(\cdot)\varphi_n(\cdot)+\1_{(0,\infty)}(-\cdot)\varphi(-\cdot)) = S(\varphi_n)\to l_{\omega}^{\rm odd},
	\\
	&2K(\1_{(0,\infty)}\varphi_n)=K(\1_{(0,\infty)}\cdot)\varphi_n(\cdot)+\1_{(0,\infty)}(-\cdot)\varphi_n(-\cdot)) =K(\varphi_n)=0.
\end{align*}
This means that $l_{\omega}^{\odd}/2 \geq l_{\omega}$. 
\end{proof}

\begin{lemma}
There is no minimizer of $l_{\omega}^{\odd}$. 
\end{lemma}

\begin{proof}
Suppose that $f \in H_{\odd}^{1}(\mathbb{R})$ is a minimizer, i.e. $S_{\omega}(f)=l_{\omega}^{\odd}$ and $K(f)=0$. 
Then we have
\begin{align*}
	&2S(\1_{(0,\infty)} f)=S(\1_{(0,\infty)}(\cdot)f(\cdot)-\1_{(0,\infty)}(-\cdot)f(-\cdot))=S(f) = l_{\omega}^{\rm odd}=2l_{\omega},
	\\
	&2K(\1_{(0,\infty)}f)=K(\1_{(0,\infty)}(\cdot)f(\cdot)-\1_{(0,\infty)}(-\cdot)f(-\cdot))=K(f)=0.
\end{align*}
The uniqueness of the minimizer of $l_{\omega}$ implies that $\1_{(0,\infty)}f=e^{i\theta}Q_{\omega}(\cdot-y)$. Obviously, this is a contradiction. 
\end{proof}
\subsubsection{Energy and mass condition}

In this section, we prove that the action condition $S_{\omega}(f)\leq 2S_{\omega}(Q_{\omega})$ is equivalent to the energy and mass condition $E(f)M(f)^{\sigma} \leq 2^{1+\sigma}E(Q)M(Q)^{\sigma}$ for odd functions $f$. 

By properties of the ground state, we have the following:
\begin{lemma}[The Pohozaev identity]
We have 
\begin{align*}
	\frac{1}{p+3}\| Q\|_{L^2}^2 = \frac{1}{p-1}\|\partial_x Q\|_{L^2}^2 = \frac{1}{2(p+1)}\| Q\|_{L^{p+1}}^{p+1} 
\end{align*}
and, in particular, 
\begin{align*}
	M(Q)= \frac{2(p+3)}{p-5} E(Q).
\end{align*}
\end{lemma}

\begin{lemma}[Scaling of the ground state]
We have 
\begin{align*}
	M(Q_{\omega}) = \omega^{-\frac{p-5}{2(p-1)}}M(Q),
	\ 
	E(Q_{\omega}) = \omega^{\frac{p+3}{2(p-1)}}E(Q).
\end{align*}
\end{lemma}

\begin{proposition}
For $f \in H_{\odd}^{1}(\mathbb{R})$, the following are equivalent:
\begin{enumerate}
\item $E(f)M(f)^{\sigma} \leq 2^{1+\sigma}E(Q)M(Q)^{\sigma}$.
\item There exists $\omega>0$ such that $S_{\omega}(f)\leq 2S_{\omega}(Q_{\omega})$. 
\end{enumerate}
\end{proposition}

\begin{proof}
For a positive number $M>0$, let $F(M):=2^{1+\sigma}E(Q)M(Q)^{\sigma}M^{-\sigma}$. 
The function $F$ is strictly convex on $(0,\infty)$. By a simple calculation (using the Pohozaev identity and the scaling property), we see that $E=-\frac{\omega}{2} M + 2S_{\omega}(Q_{\omega})$ is a tangent line of $F$ at $M=2\omega^{-\frac{p-5}{2(p-1)}}M(Q)=M(Q_{\omega})$. This implies the result. 
\end{proof}

As a corollary, we see the following: 
\begin{corollary}
\label{cor2.6.0}
For $f \in H_{\odd}^{1}(\mathbb{R})$, the following are equivalent: 
\begin{enumerate}
\item $E(f)M(f)^{\sigma} = 2^{1+\sigma}E(Q)M(Q)^{\sigma}$.
\item There exists $\omega>0$ such that $E(f)=2E(Q_{\omega})$ and $M(f)=2M(Q_{\omega})$.
\end{enumerate}
\end{corollary}

\subsubsection{Reduction by the scaling}

As seen in Corollary \ref{cor2.6.0}, to prove the main result, we may assume that $M(u_0)=2M(Q_{\omega})$ and $E(u_0)=2E(Q_{\omega})$ for some $\omega>0$. By using the scaling invariance of the equation \eqref{NLS}, we may also assume that $M(u_0)=2M(Q)$ and $E(u_0)=2E(Q)$. We prove the sufficiency in this section. 

\begin{theorem}
\label{thm4.7}
Let $u_0 \in H_{\odd}^{1}(\mathbb{R})$ satisfy $M(u_0)=2M(Q)$ and $E(u_0)=2E(Q)$. Then we have
\begin{enumerate}
\item If $K(u_0)> 0$, then the solution scatters in both time directions.
\item In addition, we assume finite variance. 
If $K(u_0)<0$, then the solution blows up in finite positive and negative time.
\end{enumerate}
\end{theorem}

If the above Theorem \ref{thm4.7} holds, we get the main result, Theorem \ref{thm1.2}, by a scaling argument.

\begin{proof}[Proof of Theorem \ref{thm1.2} from Theorem \ref{thm4.7}]
Let $u_0 \in H_{\odd}^{1}(\mathbb{R})$ satisfy $E(u_0)M(u_0)^{\sigma}=2^{1+\sigma}E(Q)M(Q)^{\sigma}$ and $K(u_0)> 0$. 
We have $M(u_0)>0$ by assumption. Thus, there exists $\omega>0$ such that $M(u_0)=\omega^{\frac{p-5}{2(p-1)}}M(Q)=M(Q_{\omega})$. Then since we also have
\begin{align*}
	2^{1+\sigma}E(Q)M(Q)^{\sigma} 
	= 2^{1+\sigma} \omega^{-\frac{p+3}{2(p-1)}}E(Q)\{\omega^{\frac{p-5}{2(p-1)}}M(Q)\}^{\sigma}
	=2^{1+\sigma}E(Q_{\omega})M(Q_{\omega})^{\sigma},
\end{align*}
we get
\begin{align*}
	E(u_0)=2E(Q_{\omega})=2\omega^{-\frac{p+3}{2(p-1)}}E(Q).
\end{align*}
Therefore, we obtain
\begin{align*}
	\omega^{\frac{p+3}{2(p-1)}}E(u_0)=2E(Q)
	\text{ and }
	\omega^{-\frac{p-5}{2(p-1)}}M(u_0)=M(Q).
\end{align*}
Let $u_{0,\omega^{-1}}(x):=\omega^{-1/(p-1)}u_{0}(\omega^{-1/2}x)$. Then we get
\begin{align*}
	E(u_{0,\omega^{-1}})=2E(Q)
	\text{ and }
	M(u_{0,\omega^{-1}})=M(Q).
\end{align*}
We also have
\begin{align*}
	K(u_{0,\omega^{-1}})=\omega^{\frac{p+3}{2(p-1)}}K(u_0)> 0
\end{align*}
By Theorem \ref{thm4.7}, we find that the solution $u_{\omega^{-1}}$ with the initial data $u_{0,\omega^{-1}}$ scatters. 
Since the equation \eqref{NLS} is invariant under this scaling, the behavior of the solution $u$ with $u(0)=u_0$  is same as that of  $u_{\omega^{-1}}$. Thus $u$ scatters. This argument also works for the case that $K$ is  negative. 
\end{proof}


Thus, it is enough to consider the case of $\omega=1$. 


\subsubsection{The virial functional and variational argument}

We will show the following in two ways:
\begin{proposition}
\label{prop4.8}
Let $f\in H_{\odd}^{1}(\mathbb{R})$. 
Assume $M(f)= 2M(Q)$ and $E(f)=2 E(Q)$. 
Then the following are equivalent:
\begin{enumerate}
\item $K(f)> 0$.
\item $\|f\|_{L^2}^{2\sigma} \|\partial_x f\|_{L^2}^{2} < 2^{1+\sigma}\|Q\|_{L^2}^{2\sigma} \|\partial_x Q\|_{L^2}^{2} $.
\item $\mu(f)> 0$.
\end{enumerate}
\end{proposition}

The equivalence of (2) and (3) holds by $M(f)= 2M(Q)$. The equivalence of (1) and (3) follows from $E(f)=2 E(Q)$:

\begin{lemma}
\label{lem2.11}
If $E(f)=2E(Q)$, we have $K(f) = \frac{p-5}{4} \mu(f)$.
\end{lemma}
\begin{proof}
We have
\begin{align*}
	K(f) &= \frac{p-1}{2} E(f) -\frac{p-5}{4} \|\partial_x f\|_{L^2}^2
	\\
	&= \frac{p-1}{2} 2E(Q) -\frac{p-5}{4} \|\partial_x f\|_{H}^2
	\\
	&= \frac{p-5}{4} \left(2\|\partial_x Q\|_{L^2}^2 -\|\partial_x f\|_{L^2}^2\right)
	\\
	&= \frac{p-5}{4} \mu(f),
\end{align*}
by the Pohozaev identity $E(Q)=\frac{p-5}{2(p-1)}\|\partial_x Q\|_{L^2}^2$.
\end{proof}

We also give a direct proof of the equivalence of (1) and (2) by using the best constant of the Gagliardo-Nirenberg inequality for odd functions. The best constant will be also used later.

\begin{lemma}[The Gagliardo-Nirenberg inequality for odd functions]
\label{GN}
We have
\begin{align*}
	\|f\|_{L^{p+1}}^{p+1}
	\leq C_{GN}^{\odd} \|f\|_{L^2}^{\frac{p+3}{2}} \|\partial_x f \|_{L^2}^{\frac{p-1}{2}} 
\end{align*}
for any odd functions $f \in H^1(\mathbb{R})$
and 
\begin{align*}
	C_{GN}^{\odd}:= \sup \left\{ \frac{\|f\|_{L^{p+1}}^{p+1}}{ \|f\|_{L^2}^{\frac{p+3}{2}} \|\partial_x f \|_{L^2}^{\frac{p-1}{2}} }: f\in H_{\odd}^{1}(\mathbb{R}) \setminus \{0\} \right\}=2^{-\frac{p-1}{2}}C_{GN}
\end{align*}
\end{lemma}

\begin{proof}
By the usual Gagliardo-Nirenberg inequality, we have $C_{GN}^{\odd}\leq C_{GN}$. 
We show $C_{GN}^{\odd}\geq 2^{-\frac{p-1}{2}}C_{GN}$. 
Applying the Gagliardo-Nirenberg inequality to $\varphi_{n}:=\mathcal{T}_{n}Q-\mathcal{T}_{-n}Q$, we get
\begin{align*}
	\|\varphi_{n}\|_{L^{p+1}}^{p+1}
	\leq C_{GN}^{\odd} \|\varphi_{n}\|_{L^2}^{\frac{p+3}{2}} \|\partial_x \varphi_{n} \|_{L^2}^{\frac{p-1}{2}} 
\end{align*}
Taking the limit $n \to \infty$, we obtain
\begin{align*}
	\|Q\|_{L^{p+1}}^{p+1}
	\leq 2^{\frac{p-1}{2}}C_{GN}^{\odd} \|Q\|_{L^2}^{\frac{p+3}{2}} \|\partial_xQ \|_{L^2}^{\frac{p-1}{2}}
\end{align*}
Therefore, we have $C_{GN}^{\odd}\geq 2^{-\frac{p-1}{2}}C_{GN}$. 

Next, we show $2^{-\frac{p-1}{2}}C_{GN} \geq C_{GN}^{\odd}$. Let $\{f_n\}\subset H_{\odd}^{1}\setminus\{0\}$ satisfy
\begin{align*}
	 \frac{\|f_n\|_{L^{p+1}}^{p+1}}{ \|f_n\|_{L^2}^{\frac{p+3}{2}} \|\partial_x f_n \|_{L^2}^{\frac{p-1}{2}} } \to C_{GN}^{\odd}
\end{align*}
We have
\begin{align*}
	(\text{LHS})=2^{-\frac{p-1}{2}}\frac{\|\1_{(0,\infty)} f_n\|_{L^{p+1}}^{p+1}}{ \|\1_{(0,\infty)} f_n\|_{L^2}^{\frac{p+3}{2}} \|\partial_x(\1_{(0,\infty)}  f_n) \|_{L^2}^{\frac{p-1}{2}} }
	\leq 2^{-\frac{p-1}{2}} C_{GN}
\end{align*}
Since the left hand side goes to $(C_{GN}^{\odd})^{-1}$, we get $2^{-\frac{p-1}{2}}C_{GN} \geq C_{GN}^{\odd}$. 
This completes the proof.
\end{proof}

\begin{lemma}
\label{lem4.9}
Let $f\in H_{\odd}^1(\mathbb{R})$. 
If $\|f\|_{L^2}^{2\sigma} \|\partial_x f\|_{L^2}^2 \leq 2^{1+\sigma}\|Q\|_{L^2}^{2\sigma} \|\partial_x Q\|_{L^2}^2$, then we have $K(f) \geq 0$.
Moreover, if $M(f)^{\sigma} E(f) \leq 2^{1+\sigma}M(Q)^{\sigma}E(Q)$ and $K(f) \geq 0$, then we have $\|f\|_{L^2}^{2\sigma} \|\partial_x f\|_{L^2}^2 \leq 2^{1+\sigma}\|Q\|_{L^2}^{2\sigma} \|\partial_x Q\|_{L^2}^2$. 
\end{lemma}

\begin{proof}
We assume that $\|f\|_{L^2}^{\sigma} \|\partial_x f\|_{L^2} \leq 2^{\frac{1+\sigma}{2}}\|Q\|_{L^2}^{\sigma} \|\partial_x Q\|_{L^2}$.
This means that $\|f\|_{L^2}^{\frac{p+3}{2}} \|\partial_x f\|_{L^2}^{\frac{p-5}{2}} \leq 2^{\frac{p-1}{2}} \|Q\|_{L^2}^{\frac{p+3}{2}} \|\partial_x Q\|_{L^2}^{\frac{p-5}{2}}$. 
By the Gagliardo-Nirenberg inequality for odd functions and the assumption, we have
\begin{align*}
	K(f)
	&\geq \|\partial_x f\|_{L^2}^2  -\frac{(p-1)C_{GN}^{\odd}}{2(p+1)} \|f\|_{L^2}^{\frac{p+3}{2}} \|\partial_x f \|_{L^2}^{\frac{p-1}{2}} 
	\\
	&=\|\partial_x f\|_{L^2}^2 \left(1 - \frac{(p-1)C_{GN}^{\odd}}{2(p+1)} \|f\|_{L^2}^{\frac{p+3}{2}} \|\partial_x f\|_{L^2}^{\frac{p-5}{2}}  \right)
	\\
	&\geq \|\partial_x f\|_{L^2}^2 \left(1 - \frac{(p-1)C_{GN}^{\odd}}{2(p+1)} 2^{\frac{p-1}{2}} \|Q\|_{L^2}^{\frac{p+3}{2}} \|\partial_x Q\|_{L^2}^{\frac{p-5}{2}}  \right)
\end{align*}
By the Pohozaev identity, we get
\begin{align*}
	\frac{(p-1)C_{GN}^{\odd}}{2(p+1)} 2^{\frac{p-1}{2}} \|Q\|_{L^2}^{\frac{p+3}{2}} \|\partial_x Q \|_{L^2}^{\frac{p-5}{2}}  =1.
\end{align*}
Therefore, we have $K(f) \geq  0$. 

Next, assume $K(f)\geq 0$. By the Pohozaev identity and $2^{1+\sigma}M(Q)^{\sigma} E(Q) \geq M(f)^{\sigma}E(f)$, we obtain
\begin{align}
\label{eq1}
	\frac{p-5}{2(p-1)}2^{1+\sigma} M(Q)^{\sigma} \|\partial_x Q\|_{L^2}^2 
	=2^{1+\sigma} M(Q)^{\sigma} E(Q) \geq M(f)^{\sigma}E(f). 
\end{align}
By the assumption, we have
\begin{align}
\label{eq1.2}
	\notag
	E(f) \geq 
	E(f) - \frac{2}{p-1}K(f)
	&=\frac{p-5}{2(p-1)} \|\partial_x f \|_{L^2}^2.
\end{align}
Therefore, combining this with \eqref{eq1}, we get
\begin{align*}
	\frac{p-5}{2(p-1)}2^{1+\sigma} M(Q)^{\sigma} \|\partial_x Q \|_{L^2}^2  
	\geq  \frac{p-5}{2(p-1)} M(f)^{\sigma} \|f\|_{H}^2.
\end{align*}
This completes the proof. 
\end{proof}

Lemma \ref{lem4.9} includes the equivalence between (1) and (2) in Proposition \ref{prop4.8}.


%

The following corollary immediately holds. 

\begin{corollary}
Let $f\in H_{\odd}^{1}(\mathbb{R})$. 
Assume $M(f)= 2M(Q)$ and $E(f)=2 E(Q)$. Then the following are equivalent:
\begin{enumerate}
\item $K(f) <  0$.
\item $\|f\|_{L^2}^{2\sigma} \|\partial_x f\|_{L^2}^{2} > 2^{1+\sigma}\|Q\|_{L^2}^{2\sigma} \|\partial_x Q\|_{L^2}^{2} $.
\item $\mu(f)< 0$.
\end{enumerate}
\end{corollary}

\subsubsection{Invariance of the potential-well sets by the flow}

\begin{proposition}
\label{prop2.13}
Let $u_0 \in H_{\odd}^1(\mathbb{R})$ satisfy $M(u_0)=2M(Q)$ and $E(u_0)=2E(Q)$ and $u(t)$ be a solution to \eqref{NLS} with $u(0)=u_0$. Then we have the following.
\begin{enumerate}
\item If $K(u_0)> 0$, then the solution $u(t)$ exists globally in both time directions and $K(u(t))>0$ for all $t \in \mathbb{R}$. 
\item If $K(u_0)<0$, then the solution $u(t)$ satisfies $K(u(t))<0$ 
while the solution exists. 
\end{enumerate}
\end{proposition}

\begin{proof}
Since $u_0$ is odd, the solution $u(t)$ is also odd. 
We consider the case of $K(u_0)> 0$. Suppose that there exists a time $t_*$ such that $K(u(t_*))\leq 0$. Then by the continuity of the flow, there exists a time $t_0$ such that $K(u(t_0))=0$. This and the assumption $M(u_0)=2M(Q)$ and $E(u_0)=2E(Q)$ mean that $u(t_0)$, which is not the zero function, is a minimizer of $l^{\odd}$. However, this contradicts the non-existence of the minimizer. Therefore we have $K(u(t)) > 0$ for all $t$ in the existence interval. 
We get an a priori bound by Proposition \ref{prop4.8}, and the blow-up alternative implies that the solution exists globally in both time directions. In the case of $K(u_0)< 0$, we get the result in the same way. 
\end{proof}

\begin{remark}
By Propositions \ref{prop4.8} and \ref{prop2.13}, solutions $u(t)$ with  $M(u_0)=2M(Q)$, $E(u_0)=2E(Q)$ and $K(u_0)>0$ are uniformly bounded in $H^1(\mathbb{R})$. 
\end{remark}



\subsection{Lemmas}

We collect some useful lemmas such as the estimate of the ground state, long time perturbation and linear profile decomposition. 


\subsubsection{Estimate for the ground state}
By the explicit formula \eqref{exp} of the ground state $Q$, we have $Q(x)\approx e^{-|x|}$. Thus, the following estimates hold by direct calculations: 

\begin{lemma}
\label{lem2.14}
Let $\alpha,\beta>0$ and $y>0$. We have
\begin{align*}
	\int_{\mathbb{R}} \mathcal{T}_{y}Q(x)^{\alpha} \mathcal{T}_{-y}Q(x)^{\beta} dx
	\approx 
	\begin{cases}
	(1+y)  \exp \left( - 2\alpha   y \right),
	& \text{ if } \alpha=\beta,
	\\
	\exp \left( -2\min\{\alpha,\beta\} y \right),
	& \text{ if } \alpha \neq \beta.
	\end{cases}
\end{align*}
\end{lemma}

\begin{proof}
This follows from $Q(x)\approx e^{-|x|}$ and direct calculations. 
\end{proof}

\begin{lemma}
\label{lem2.15}
Let $\alpha,\beta\geq 0$. Let $y>0$. If $\alpha \neq 0$ or $\beta\neq 0$, then we have
\begin{align*}
	\int_{0}^{\infty} \mathcal{T}_{y}Q(x)^{\alpha} \mathcal{T}_{-y}Q(x)^{\beta} dx
	\approx 
	\begin{cases}
	(1+y) \exp \left(  -2\alpha y \right), & \text{ if } \alpha=\beta,
	\\
	\exp \left\{  -(\min\{\alpha,\beta\}+\beta) y \right\}, & \text{ if } \alpha \neq \beta.
	\end{cases}
\end{align*}
\end{lemma}

\subsubsection{Virial identity and its localization}

For a solution $u(t)$, we define 
\begin{align*}
	J(u(t))&=J_{\infty}(u(t)):=\int_{\mathbb{R}} |x|^2 |u(t,x)|^2 dx.
\end{align*}
Then we have
\begin{align*}
	J'(u(t))&=2\im \int_{\mathbb{R}} x \overline{u(t,x)} \partial_x u(t,x) dx,
	\\
	J''(u(t))&=8K(u(t)).
\end{align*}

Let $\varphi$ be an even function in $C_{0}^{\infty}(\mathbb{R})$ satisfying
\begin{align*}
	\varphi(x) :=
	\begin{cases}
	x^2, & (|x|<1),
	\\
	0, & (|x|>2).
	\end{cases}
\end{align*}
For a solution $u(t)$, we set
\begin{align*}
	J_R(u(t))&:= \int_{\mathbb{R}} R^2 \varphi\left( \frac{x}{R}\right) |u(t,x)|^2 dx.
\end{align*}
Then, we have
\begin{align*}
	J_R'(u(t))&=2\im \int_{\mathbb{R}} R (\partial_x \varphi)\left( \frac{x}{R}\right) \overline{u(t,x)} \partial_x u(t,x) dx
\end{align*}
and 
\begin{align*}
	J_R''(u(t))&=4\int_{\mathbb{R}} (\partial_x^2 \varphi)\left( \frac{x}{R}\right) \left\{ |\partial_x u(t,x)|^2 - \frac{p-1}{2(p+1)} |u(t,x)|^{p+1} \right\} dx
	\\
	&\quad - \int_{\mathbb{R}} \frac{1}{R^2} (\partial_x^4 \varphi) \left( \frac{x}{R}\right) |u(t,x)|^2 dx
	\\
	&=8K(u(t)) +A_R(u(t)),
\end{align*}
where we set  
\begin{align*}
	A_R(u(t))&:=
	- 4\int_{|x|>R} \left\{2 -(\partial_x^2 \varphi)\left( \frac{x}{R}\right)\r\} \left\{ |\partial_x u(t,x)|^2 - \frac{p-1}{2(p+1)} |u(t,x)|^{p+1} \right\} dx
	\\
	&\quad - \int_{R<|x|<2R} \frac{1}{R^2} (\partial_x^4 \varphi) \left( \frac{x}{R}\right) |u(t,x)|^2 dx.
\end{align*}

For a function $f \in H^1$, we set 
\begin{align*}
	F_R(f)&:=4\int_{\mathbb{R}} (\partial_x^2 \varphi)\left( \frac{x}{R}\right) \left\{ |\partial_x f|^2 - \frac{p-1}{2(p+1)} |f|^{p+1} \right\} dx
	 - \int_{\mathbb{R}} \frac{1}{R^2} (\partial_x^4 \varphi) \left( \frac{x}{R}\right) |f|^2 dx.
\end{align*}
Then the above equality is written by
\begin{align*}
	A_R(u(t))=F_R(u(t))-8K(u(t)) = F_R(u(t)) - F_{\infty}(u(t)),
\end{align*}
where we note that $F_{\infty}(f)=8K(f)$. 

\begin{lemma}
\label{lem2.16}
We have $K(e^{i\theta}\mathcal{T}_{y}Q)=0$ and thus $8^{-1}F_{\infty}(e^{i\theta}\mathcal{R}_{y}Q)=K(e^{i\theta}\mathcal{R}_{y}Q)=O((1+y)e^{-2y})$ for any $\theta \in \mathbb{R}$ and $y>0$.
\end{lemma}

\begin{proof}
$K(e^{i\theta}\mathcal{T}_{y}Q)=0$ follows from the Pohozaev identity. The latter statement follows from the first statement and  the estimate of cross terms by Lemma \ref{lem2.14}.
\end{proof}

\begin{lemma}
We have $F_R(e^{i\theta}\mathcal{T}_{y}Q)=0$ and thus $F_{R}(e^{i\theta}\mathcal{R}_{y}Q)=O((1+y)e^{-2y})$ for any $\theta \in \mathbb{R}$ and $y>0$.
\end{lemma}

\begin{proof}
The first statement follows from \cite[Lemma 2.9]{MMZ21}. The latter statement holds as in Lemma \ref{lem2.16}. 
\end{proof}


\subsubsection{Strichartz estimates and long time perturbations}
We use the following Lebesgue exponents:
\begin{align*}
	r:=p+1, \quad a:=\frac{2(p-1)(p+1)}{p+3}
	\quad \text{and}\quad 
	b:=\frac{2(p-1)(p+1)}{(p-1)^{2}-(p-1)-4}.
\end{align*}
We have the following Strichartz estimates (see \cite{CW92,Caz03,FXC11}): 
\begin{align*}
	&\|e^{it\partial_x^2}f\|_{L^{a}_{t}L^{r}_{x}}\lesssim \|f\|_{H^{1}},
	\\
	&\|\int^{t}_{0}e^{i(t-s)\partial_x^2}g(s)ds\|_{L^{a}_{t}L^{r}_{x}}\lesssim \|g\|_{L^{b^{\prime}}_{t}L^{r^{\prime}}_{x}}.
\end{align*}

In the following result we have a sufficient condition for scattering.

\begin{proposition}
\label{prop}
Let $u_{0}\in H^{1}(\mathbb{R})$ and $u$ be the corresponding solution to \eqref{NLS} with initial data $u(0)=u_{0}$.
If $u\in L^{a}_{t}L^{r}_{x}((0,\infty)\times\mathbb{R})$, then $u$ scatters in the positive time direction.
\end{proposition}

It is known that the Cauchy problem is well-posed. We also see that the  final value problem is well-posed, that is there exists a wave operator. 

\begin{proposition}[{\cite[Theorem 3]{Str81}}]
Let $\psi \in H^{1}(\mathbb{R})$. There exists $T \in \mathbb{R}$ and the solution $u(t)$ to \eqref{NLS} on $(T,\infty)$ such that 
\begin{align*}
	\|u(t) - e^{it\partial_x^2} \psi\|_{H^1} \to 0 \text{ as } t \to \infty.
\end{align*}
\end{proposition}

We will use this long time perturbation result (see \cite{FXC11} and \cite{AkNa13} for the proof). 

\begin{lemma}[Long time perturbation]
\label{perturb}
For any $M>0$, there exist $\varepsilon=\varepsilon(M)>0$ and a positive constant $C=C(M)$ such that the following occurs. 
Let $v: I\times \mathbb{R}\to \mathbb{C}$ be a solution of the integral equation with source term $e$:
\[ v(t)=e^{it\partial_x^2}\varphi +i\int_{0}^{t} e^{i(t-s)\partial_x^2} (|v(s)|^{p-1}v(s))ds +e (t)\]
with $\|v\|_{L_t^a L_x^r(I\times \R)}<M$ and $\|e\|_{L_t^{a}L_x^r(I\times \R)}<\varepsilon$. Assume moreover that $u_0 \in H^1(\mathbb{R})$ is such that $\|u_0-\varphi\|_{H^{1}} < \varepsilon$, then the solution $u: I\times \R\to \C$ to \eqref{NLS} with initial data $u_0$:
\[ u(t)=e^{it\partial_x^2}u_{0} +i \int_{0}^{t} e^{i(t-s)\partial_x^2}(|u(s)|^{p-1}u(s)) ds,\]
satisfies $u \in L_t^{a}L_x^r(I\times \R)$ and moreover $\|u-v\|_{L_t^{a}L_x^r(I\times \R)}<C\varepsilon$. 
\end{lemma}


\subsubsection{Linear profile decomposition}

We have a linear profile decomposition for odd functions. 

\begin{proposition}[Linear profile decomposition for odd functions] 
\label{LPD}
Let $\{ \varphi_n\}_{n\in \N}$ be a bounded sequence in $H_{\odd}^1(\mathbb{R})$. Then, up to subsequence, we can write 
\begin{align}
\label{lpd}
	\varphi_n=\sum_{j=1}^{J} e^{-i t_{n}^j \partial_x^2} \mathcal{R}_{x_n^j} \psi^j +R_n^J, \quad \forall J \in \N, 
\end{align}
where 
$t_n^j\in \mathbb{R}$, $x_n^j\geq 0$, $\psi^j \in H^1(\mathbb{R})$, $R_{n}^J \in H_{\odd}^{1}(\mathbb{R})$, and the following hold:
\begin{itemize}
\item for any fixed $j$, we have :
\begin{align*}
&\text{either } t_n^j=0 \text{ for any } n \in \N, \text{ or } t_n^j \to \pm \infty \text{ as } n\to \infty,
\\
&\text{either } x_n^j=0 \text{ for any } n \in \N, \text{ or } x_n^j \to \infty \text{ as } n\to \infty.
\end{align*}
\item orthogonality of the parameters:
\[ |t_n^j -t_n^k|+|x_n^j-x_n^k| \to \infty \text{ as } n \to \infty, \quad \forall j\neq k. \]
\item smallness of the remainder:
\[ \forall \eps >0, \exists J=J(\eps) \in \N \text{ such that } \limsup_{n\to \infty} \| e^{it\partial_x^2}R_n^J\|_{L_{t,x}^\infty} <\eps. \]
\item orthogonality in norms: for any $ J\in \N$
\begin{align*} 
\| \varphi_n\|_{L^2}^2 &=\sum_{j=1}^J \|\mathcal{R}_{x_n^j}\psi^j\|_{L^2}^2 +\| R_n^J\|_{L^2}^2 +o_n(1),
\\
\| \partial_x \varphi_n\|_{L^2}^2 &=\sum_{j=1}^J \|\partial_x \mathcal{R}_{x_n^j}\psi^j \|_{L^2}^2 +\|\partial_x R_n^J\|_{L^2}^2 +o_n(1).
\end{align*}
Moreover, we have 
\[ \| \varphi_n\|_{L^q}^q =\sum_{j=1}^J \| e^{-i t_n^j \partial_x^2} \mathcal{R}_{x_n^j} \psi^j \|_{L^q}^q +\| R_n^J\|_{L^q}^q +o_n(1), \quad q\in (2,\infty),   \quad \forall J\in \N. \]
\end{itemize}
\end{proposition}

\begin{proof}
We give a sketch of the proof. See also \cite[Theorem 3.5]{IkIn17}. By the linear profile decomposition, e.g. \cite[Theorem 5.1.]{FXC11}, we have
\begin{align*}
	\varphi_n=\sum_{j=1}^{J} e^{-i t_{n}^j \partial_x^2} \mathcal{T}_{x_n^j} \tilde{\psi}^j +\tilde{R}_n^J.
\end{align*}
Since $\varphi_n$ is odd, we have
\begin{align*}
	\varphi_n(x) 
	&=\frac{\varphi_n(x)+\varphi_n (-x)}{2}
	\\
	&= \sum_{j=1}^{J} e^{-i t_{n}^j \partial_x^2}  \frac{(\mathcal{T}_{x_n^j} \tilde{\psi}^j)(x) + (\mathcal{T}_{x_n^j} \tilde{\psi}^j)(-x)}{2}
	 +\frac{\tilde{R}_n^J(x)+\tilde{R}_n^J(-x)}{2}.
\end{align*}
Setting $\psi^j:=\tilde{\psi}^j/2$ and $R_n^J:=(\tilde{R}_n^J(x)+\tilde{R}_n^J(-x))/2$, we get \eqref{lpd} and the other statements follows from the linear profile decomposition. 
See \cite{Inui17} for a direct proof. 
\end{proof}

\begin{remark}
\label{rmk1}
The $L_t^a L_x^r$-norm of $e^{it\partial_x^2}R_n^J$ is controlled by its $L_{t,x}^{\infty}$-norm. See \cite[Remark 2.3]{ArIn21}. 
\end{remark}


%
%
%
%

\section{Coercivity of the linearized operator}

In this section, we give the coercivity of a linearized operator, which is used in the following modulation argument. Before stating our coercivity, we recall the coercivity statement of \cite{CFR20}
for the linearized equation. 

We set 
\begin{align*}
	\Phi(f,g)
	&=\re \int_{\mathbb{R}} \partial_x f(x)  \overline{\partial_x g(x)}  
	+ f(x) \overline{g(x)}
	\\
	&\quad -|Q|^{p-1}(pf_1(x)g_1(x) + f_2(x)g_2(x)) dx,
\end{align*}
where $f=f_1+if_2$ and $g=g_1+i g_2$ and we also set $\Phi(f) :=\Phi(f,f) $. 

\begin{lemma}[{\cite[Lemma 3.5]{CFR20}}]
\label{lem3.1}
Let $h \in H^{1}(\mathbb{R})$ satisfy the following orthogonality conditions: 
\begin{align*}
	\im \int h \mathcal{T}_yQ dx 
	= \re \int h \partial_x (\mathcal{T}_yQ) dx 
	=\re \int h (\mathcal{T}_yQ)^{p} dx
	=0. 
\end{align*}
Then, there exists a positive constant $c$ such that
\begin{align*}
	\Phi(\mathcal{T}_{-y}h) \geq c \|h\|_{H^1}^2.
\end{align*}
\end{lemma}

See also \cite{Wei89} and \cite{DuRo10}. 

Using Lemma \ref{lem3.1}, we obtain the following coercivity
statement: 

\begin{lemma}[Coercivity]
\label{coercivity}
Let $h \in H_{\odd}^{1}(\mathbb{R})$ satisfy the orthogonality conditions: 
\begin{align*}
	\im \int h \chi_{R}^{+}  \mathcal{T}_yQ dx 
	= \re \int h \partial_x (\chi_{R}^{+} \mathcal{T}_yQ) dx 
	=\re \int h \chi_{R}^{+} (\mathcal{T}_yQ )^{p} dx
	=0.
\end{align*}
Then, there exists a positive constant $c$ such that
\begin{align*}
	\Phi(\mathcal{T}_{-y}(\chi_{R}^{+}h)) \geq c \|\chi_{R} h\|_{H^1}^2  - \frac{1}{R}\|h\|_{H^1}^2
\end{align*}
for $R>1$. 
\end{lemma}

\begin{proof}
We set
\begin{align*}
	\mathscr{A}&:=\{f \in H_{\odd}^{1}(\mathbb{R}): \re \int f Q ^{p} dx =  \im \int f Q dx = \re \int f \partial_x Q dx = 0\},
	\\
	\mathscr{B}&:= \Span_{\mathbb{R}}\{ iQ, Q^{p}, \partial_x Q\}.
\end{align*}
Then, we write $\mathcal{T}_{-y}(\chi_{R}^{+}h) = \mathcal{T}_{-y}\tilde{h} + \mathcal{T}_{-y}r$ such that $ \mathcal{T}_{-y}\tilde{h} \in \mathscr{A}$ and $ \mathcal{T}_{-y}r \in \mathscr{B}$. 
Since $ \mathcal{T}_{-y}r \in \mathscr{B}$, we write
\begin{equation}
\label{eq5.1}
	 \mathcal{T}_{-y}r =\alpha \partial_x Q + \beta i Q+\gamma Q^{p},
\end{equation}
where $\alpha,\beta,\gamma \in \mathbb{R}$. 
By multiplying \eqref{eq5.1} by $-iQ$ and taking real part and integral, we get
\begin{align*}
	\beta = \frac{1}{\|Q\|_{L^2}^2} \langle \mathcal{T}_{-y}r, iQ\rangle
	= \frac{1}{\|Q\|_{L^2}^2} \langle  \mathcal{T}_{-y}(\chi_{R}^{+}h) - \mathcal{T}_{-y}\tilde{h}  ,  iQ\rangle.
\end{align*}
We have $\langle \mathcal{T}_{-y}\tilde{h}, iQ \rangle=0$ since $\mathcal{T}_{-y}\tilde{h} \in \mathscr{A}$ and we also have $\langle  \mathcal{T}_{-y}(\chi_{R}^{+}h) ,  iQ \rangle=0$ by the orthogonality assumption. Thus, $\beta=0$. By multiplying \eqref{eq5.1} by $Q_{\omega}^{p}$ and taking integral and real part, we get $\gamma=0$ in a similar way. By multiplying \eqref{eq5.1} by $\partial_x Q$ and taking integral and imaginary part, we get
\begin{align*}
	\alpha = \frac{1}{\|\partial_x Q \|_{L^2}^2 } 
	\langle  \mathcal{T}_{-y}r ,  \partial_x Q \rangle
	= \frac{1}{\|\partial_x Q\|_{L^2}^2 } 
	\langle   \mathcal{T}_{-y}(\chi_{R}^{+}h) ,  \partial_x Q \rangle
\end{align*}
since $\mathcal{T}_{-y}\tilde{h} \in \mathscr{A}$. Now, we have
\begin{align*}
	\langle   \mathcal{T}_{-y}(\chi_{R}^{+}h) ,  \partial_x Q \rangle
	&=\re \int  \chi_{R}^{+}h \partial_x \mathcal{T}_{y}Q dx
	\\
	&=\re \int  \chi_{R}^{+}h \partial_x \mathcal{T}_{y}Q dx
	-\re \int  h\partial_x (\chi_{R}^{+} \mathcal{T}_{y}Q )dx,
\end{align*}
where we use $\re \int  h \partial_x (\chi_{R}^{+} \mathcal{T}_{y}Q)dx=0$ by the orthogonality assumption. By a direct calculation, it holds that
\begin{align*}
	\left| \re \int  \chi_{R}^{+}h \partial_x \mathcal{T}_{y}Q dx
	-\re \int  h\partial_x (\chi_{R}^{+} \mathcal{T}_{y}Q )dx\right|
	&=\int |h| |\partial_x\chi_{R}^{+}| \mathcal{T}_{y}Q dx
	\\
	&\lesssim R^{-1}\|h\|_{L^2}.
\end{align*}
Therefore, we obtain
\begin{align*}
	\|r\|_{H^1}=\|\mathcal{T}_{-y}r\|_{H^1} = |\alpha|\|\partial_x Q \|_{H^1}\lesssim \frac{1}{R}\|h\|_{L^2}
\end{align*}
and thus
\begin{align}
\label{eq3.2}
	|\Phi(\mathcal{T}_{-y}r)| \lesssim \frac{1}{R^2} \|h\|_{L^2}^2.
\end{align}
Now, since $\Phi$ is bilinear, we get
\begin{align}
\label{eq3.3}
	\Phi(\mathcal{T}_{-y}(\chi_{R}^{+}h))=\Phi(\mathcal{T}_{-y}\tilde{h}+\mathcal{T}_{-y}r)= \Phi(\mathcal{T}_{-y}\tilde{h})+\Phi(\mathcal{T}_{-y}r)+2\Phi(\mathcal{T}_{-y}\tilde{h},\mathcal{T}_{-y}r).
\end{align}
Then, by Lemma \ref{lem3.1}, we have
\begin{align}
\label{eq3.4}
	\Phi( \mathcal{T}_{-y}\tilde{h} ) \gtrsim \| \mathcal{T}_{-y}\tilde{h}\|_{H^1}^2.
\end{align}
Moreover, we have
\begin{align}
\label{eq3.5}
	\| \mathcal{T}_{-y}\tilde{h}\|_{H^1}^2
	\gtrsim \|\mathcal{T}_{-y}(\chi_{R}^{+} h)\|_{H^1}^2 - \|\mathcal{T}_{-y}r\|_{H^1}^2
	\gtrsim \|\chi_{R}^{+}h\|_{H^1}^2 -\frac{1}{R^2}\|h\|_{H^1}^2. 
\end{align}
Now, we have
\begin{align}
\label{eq3.6}
	|\Phi(\mathcal{T}_{-y}\tilde{h},\mathcal{T}_{-y}r)|
	\lesssim \frac{1}{R}\|\tilde{h}\|_{H^1}\|h\|_{H^1}
	\lesssim \frac{1}{R}\|h\|_{H^1}^2.
\end{align}
Combining \eqref{eq3.2}--\eqref{eq3.6}, we obtain
\begin{align*}
	\Phi(\mathcal{T}_{-y}(\chi_{R}^{+}h)) 
	\gtrsim \|\chi_{R}^{+}h\|_{H^1}^2 - \frac{1}{R}\|h\|_{H^1}^2.
\end{align*}
By the symmetry, we have $2 \|\chi_{R}^{+}h\|_{H^1} = \|\chi_{R}h\|_{H^1}$. 
This completes the proof. 
\end{proof}

\section{Modulation}

We start with a simple version of modulation. 


\begin{lemma}
\label{lem4.1}
There exists $\mu_{0}>0$ and a function $\varepsilon:(0,\mu_{0}) \to (0,\infty)$ with $\varepsilon(\mu) \to 0$ as $\mu \to 0$ such that the following holds. For any $\mu < \mu_0$ and for all $f\in H_{\odd}^{1}(\mathbb{R})$ satisfying $E(f)=2E(Q)$, $M(f)=2M(Q)$ and $\mu(f)<\mu$,  there exist $(\theta, y)\in \mathbb{R} \times [0,\infty)$ such that
\begin{align}
	\label{eq4.1.1}
	\| f-e^{i\theta}\mathcal{R}_{y}Q\|_{H^{1}}\leq \varepsilon(\mu).
\end{align} 
\end{lemma}

\begin{proof}
We use a contradiction argument. We suppose that the statement fails. 
%
Then there exists $\varepsilon_0>0$ such that for any $n\in \mathbb{N}$, there exist $\mu_n$ with $\mu_n \to 0$ and $f_n \in H_{\odd}^{1}(\mathbb{R})$ satisfying $E(f_n)=2E(Q)$, $M(f_n)=2M(Q)$ and $\mu(f_n)<\mu_n$ such that $ \inf_{\theta \in \mathbb{R}}\inf_{y\geq 0}\| f_n-e^{i\theta}\mathcal{R}_{y}Q\|_{H^{1}} > \varepsilon_0$. 
Since $\mu(f_n) \to 0$ as $n \to \infty$ and $E(f_n)=2E(Q)$, we have $\|f_n\|_{L^{p+1}}^{p+1} \to 2\|Q\|_{L^{p+1}}^{p+1}$ and thus
\begin{align*}
	S(f_n)  \to 2S(Q) \text{ and } K(f_n)\to 2K(Q)=0.
\end{align*}
Since $f$ is odd, we have $\1_{(0,\infty)}f \in H^1(\mathbb{R})$ and it holds that 
\begin{align*}
	S(\1_{(0,\infty)}f_n)  \to S(Q) \text{ and } K(\1_{(0,\infty)}f_n)\to K(Q)=0.
\end{align*}
Therefore, $\{ \1_{(0,\infty)}f_n\}$ is a minimizing sequence, and we obtain $(\theta,y) \in \mathbb{R}^2$ such that $\1_{(0,\infty)}f_n \to e^{i\theta}\mathcal{T}_{y}Q$ in $H^1$ by the characterization of the ground state. By the symmetry, we also have $\1_{(-\infty,0)}f_n \to -e^{i\theta}\mathcal{T}_{-y}Q$ in $H^1$. Therefore it holds that
\begin{align*}
	\| f_n-e^{i\theta}\mathcal{R}_{y}Q\|_{H^{1}} \to 0.
\end{align*}
This is a contradiction. 
\end{proof}

We have the following modulation with orthogonality conditions. 

\begin{lemma}[Modulation]
\label{lem4.2}
Let $R>0$ be sufficiently large. 
There exist $\mu_0>0$ and a function $\varepsilon:(0,\mu_{0}) \to (0,\infty)$ with $\varepsilon(\mu) \to 0$ as $\mu \to 0$ such that the following holds. For any $\mu < \mu_0$ and for all $f\in H_{\odd}^{1}(\mathbb{R})$ satisfying $E(f)=2E(Q)$, $M(f)=2M(Q)$ and $\mu(f)<\mu$,  there exist $(\tilde{\theta}, y)\in \mathbb{R} \times (R,\infty)$ such that
\begin{align*}
	\| e^{-i\tilde{\theta}}f- \mathcal{R}_{y}Q\|_{H^1} < \varepsilon(\mu)
\end{align*}
and
\begin{align}
\label{eq4.2}
	\im \int_{\mathbb{R}} g \chi_{R}^{+}\mathcal{T}_{y}Qdx=0,
	\quad \re \int_{\mathbb{R}} g \partial_x(\chi_{R}^{+} \mathcal{T}_{y}Q)dx=0
\end{align}
where $g=e^{-i\tilde{\theta}}f-\mathcal{R}_{y}Q$. 
\end{lemma}

\begin{proof}
We define 
\begin{align*}
	J(\tilde{\theta},y,v)=
	\begin{pmatrix}
	J_1(\tilde{\theta},y,v)
	\\
	J_2(\tilde{\theta},y,v)
	\end{pmatrix}:=
	\begin{pmatrix}
	\im \int_{\mathbb{R}} (e^{-i\tilde{\theta}}v - \mathcal{R}_{y}Q ) \chi_{R}^{+}\mathcal{T}_{y}Qdx
	\\
	\re \int_{\mathbb{R}} (e^{-i\tilde{\theta}}v - \mathcal{R}_{y}Q ) \partial_x(\chi_{R}^{+} \mathcal{T}_{y}Q)dx
	\end{pmatrix}
\end{align*}
for $\tilde{\theta} \in \mathbb{R}$, $y\gg R$, $v \in H_{\odd}^{1}$. Then $J(0,y,\mathcal{R}_{y}Q)=0$. We have
\begin{align*}
	\frac{\partial J_1}{\partial \tilde{\theta}}(0,y,\mathcal{R}_{y}Q)
	&=- \int_{\mathbb{R}} \mathcal{R}_{y}Q  \chi_{R}^{+}\mathcal{T}_{y}Qdx
	\\
	&=-\int_{\mathbb{R}} |\mathcal{T}_{y}Q |^2 dx+O(e^{-2y})
	\\
	&=-\|Q \|_{L^2}^2+O(e^{-2y})
\end{align*}
and 
\begin{align*}
	\frac{\partial J_2}{\partial y}(0,y,\mathcal{R}_{y}Q)
	&
	=- \int_{\mathbb{R}}\partial_x (\mathcal{T}_{y}Q +\mathcal{T}_{-y}Q ) \partial_x(\chi_{R}^{+} \mathcal{T}_{y}Q)dx
	\\
	&=\|\partial_x Q \|_{L^2}^2 +O(R^{-1} + e^{-2y})
\end{align*}
Therefore we have
\begin{align*}
	\frac{\partial J(\tilde{\theta},y,v)}{\partial(\tilde{\theta},y)}(0,y,\mathcal{R}_{y}Q)
	&=\begin{pmatrix}
	- \int_{\mathbb{R}} \mathcal{R}_{y}Q \chi_{R}^{+}\mathcal{T}_{y}Qdx
	&
	0
	\\
	0
	&
	- \int_{\mathbb{R}}\partial_x (\mathcal{T}_{y}Q +\mathcal{T}_{-y}Q ) \partial_x(\chi_{R}^{+} \mathcal{T}_{y}Q)dx
	\end{pmatrix}
	\\
	&=
	\begin{pmatrix}
	-\|Q \|_{L^2}^2+O(e^{-2y})
	&
	0
	\\
	0
	&
	\|\partial_x Q \|_{L^2}^2 +O(R^{-1} + e^{-2y})
	\end{pmatrix}
\end{align*}
This is invertible for large $R$ and $y$. By the implicit function theorem, we get a function $(\tilde{\theta},y): H_{\odd}^{1} \to \mathbb{R} \times \mathbb{R}$ such that 
\begin{align*}
	&\im \int_{\mathbb{R}} (e^{-i\tilde{\theta}(v)}v - \mathcal{R}_{y(v)}Q ) \chi_{R}^{+}\mathcal{T}_{y(v)}Qdx
	\\&=
	\re \int_{\mathbb{R}} (e^{-i\tilde{\theta}(v)}v - \mathcal{R}_{y(v)}Q ) \partial_x(\chi_{R}^{+} \mathcal{T}_{y(v)}Q)dx
	\\
	&=0.
\end{align*}
This completes the proof. 
\end{proof}

Let $u$ be an odd solution satisfying 
\begin{align}
\label{ME}
	M(u(t))=2M(Q) \text{ and } E(u(t))=2E(Q).
\end{align}
We set $I_{\mu_0}:=\{ t \in I_{\max} :  \mu(t) < \mu_0\}$, where $I_{\max}$ denotes the maximal existence time interval of the solution. 
By Lemma \ref{lem4.2}, we have $C^1$ functions $\tilde{\theta}=\tilde{\theta}(t)$ and $y=y(t)$ for $t \in I_{\mu_0}$. We set $\theta:=\tilde{\theta}-1$. We also have orthogonality conditions \eqref{eq4.2}.
We set
\begin{align}
	\label{eq4.3}
	u(t,x)
	&=e^{i\theta(t) + i t}(\mathcal{R}_{y(t)}Q(x) + g(t,x))
	\\ \notag
	&=e^{i\theta(t) + i t}(\mathcal{R}_{y(t)}Q(x) + \rho(t)\mathcal{G}_{R,y(t)}Q(x) + h(t,x)),
\end{align}
where 
\begin{align}
\label{eq4.4}
	\rho(t):=\frac{\re \int g \chi_{R}^{+} (\mathcal{T}_{y(t)}Q )^p dx}{\int (\chi_{R}^{+})^2(\mathcal{T}_{y(t)}Q )^{p+1}dx}.
\end{align}
Then it follows from \eqref{eq4.2}, \eqref{eq4.3} and \eqref{eq4.4} that
\begin{align}
\label{ortho}
	\im \int_{\mathbb{R}} h \chi_{R}^{+}\mathcal{T}_{y(t)}Qdx
	=\re \int_{\mathbb{R}} h \partial_x(\chi_{R}^{+} \mathcal{T}_{y(t)}Q)dx
	=\re \int_{\mathbb{R}} h \chi_{R}^{+} (\mathcal{T}_{y(t)}Q_{\omega})^{p}dx
	=0.
\end{align}

\subsection{Estimate of the parameters}
We give estimates of the parameters.

\begin{lemma}
We have
\begin{align*}
	&|\rho| \lesssim \|g\|_{L^2},
	\\
	&\|g\|_{H^1} \lesssim |\rho| + \|h\|_{H^1},
	\\
	&\|h\|_{H^1} \lesssim |\rho| + \|g\|_{H^1} \lesssim \|g\|_{H^1}.
\end{align*}
\end{lemma}

\begin{proof}
These follow from $g=\rho \mathcal{G}_{R,y}Q + h$ and the definition of $\rho$.  
\end{proof}

By the above lemma and $\|g\|_{H^1} \leq \mu_0$, taking $\mu_0$ sufficiently small, we may assume that $\|h\|_{H^1}$ and $|\rho|$ are sufficiently small.

\begin{lemma}
We have
\begin{align*}
	|\rho| \lesssim  \|h\|_{H^1}+(1+y)e^{-2y}
\end{align*}
\end{lemma}

\begin{proof}
We have
\begin{align*}
	M(u(t))
	&=M(\mathcal{R}_{y}Q +\rho \mathcal{G}_{R,y}Q +h)
	\\
	&=2\|Q\|_{L^2}^2 
	-2 \int \mathcal{T}_{y}Q\mathcal{T}_{-y}Q dx
	\\
	&\quad +2\re \int \mathcal{R}_{y}Q (\rho \mathcal{G}_{R,y}Q +h) dx + \|\rho \mathcal{G}_{R,y}Q +h\|_{L^2}^2.
\end{align*}
It holds from $M(u(t))=2M(Q )$ that
\begin{align}
\label{eq6.4}
	-2 \int \mathcal{T}_{y}Q \mathcal{T}_{-y}Q dx
	+2\re \int \mathcal{R}_{y}Q (\rho \mathcal{G}_{R,y}Q+h) dx + \|\rho \mathcal{G}_{R,y}Q +h\|_{L^2}^2=0.
\end{align}
This implies that
\begin{align*}
	|\rho| \lesssim  \|h\|_{H^1}+(1+y)e^{-2y}
\end{align*}
\end{proof}

\begin{lemma}
\label{lem4.5}
We have
\begin{align*}
	|\rho| \approx  |\mu(t)| +O(e^{-2y} +\|h\|_{H^1}^2).
\end{align*}
\end{lemma}

\begin{proof}
By a direct calculation, we have
\begin{align}
\label{eq4.8.0}
	\mu(t)
	&=2 \int \partial_x \mathcal{T}_{y}Q  \partial_x \mathcal{T}_{-y}Q dx 
	-2\rho \re \int \partial_x (\mathcal{R}_{y}Q) \partial_x  (\mathcal{G}_{R,y}Q )dx
	\\ \notag
	&\quad -2 \re \int \partial_x (\mathcal{R}_{y}Q ) \partial_x hdx
	-\| \partial_x (\rho \mathcal{G}_{R,y}Q+ h )\|_{L^2}^2.
\end{align}
We note that
\begin{align}
\label{EEQ}
	-\partial_x^2 (\mathcal{R}_{y}Q) = - \mathcal{R}_{y}Q + (\mathcal{T}_{y}Q )^p -  (\mathcal{T}_{-y}Q )^p
\end{align}
since $Q$ satisfies the elliptic equation. 
It holds from integration by parts and \eqref{EEQ} that
\begin{align}
\label{eq4.10}
	\re \int \partial_x (\mathcal{R}_{y}Q) \partial_x hdx
	=- \re \int \mathcal{R}_{y}Q hdx 
	+ \re \int \{(\mathcal{T}_{y}Q )^p -  (\mathcal{T}_{-y}Q)^p\}hdx.
\end{align}
The second term of \eqref{eq4.10} is estimated as follows.
\begin{align}
\label{eq4.11}
	\left| \re \int \{(\mathcal{T}_{y}Q)^p -  (\mathcal{T}_{-y}Q )^p\}hdx \right| 
	\lesssim e^{-py} \|h\|_{H^1}.
\end{align}
Indeed, we have
\begin{align}
\label{eq6.5.0}
	&\re \int \{(\mathcal{T}_{y}Q )^p -  (\mathcal{T}_{-y}Q )^p\}hdx
	\\ \notag
	&=\re \int \{ \chi_{R}^{+}(\mathcal{T}_{y}Q )^p -  \chi_{R}^{-}(\mathcal{T}_{-y}Q )^p\}hdx
	\\ \notag
	&\quad +\re \int \{ (1-\chi_{R}^{+})(\mathcal{T}_{y}Q )^p - ( 1-\chi_{R}^{-})(\mathcal{T}_{-y}Q )^p\}hdx.
\end{align}
The first term in the right hand side is $0$ by \eqref{ortho} and the symmetry. It follows from Lemma \ref{lem2.15} that the second term is estimated by
\begin{align}
\label{eq6.6.0}
	\left| \re \int \{ (1-\chi_{R}^{+})(\mathcal{T}_{y}Q )^p - ( 1-\chi_{R}^{-})(\mathcal{T}_{-y}Q )^p\}hdx\right|
	\lesssim_R e^{-py} \|h\|_{H^1}. 
\end{align}
Thus we get
\begin{align}
\label{eq4.15}
	\re \int \partial_x (\mathcal{R}_{y}Q) \partial_x hdx
	=- \re \int \mathcal{R}_{y}Q hdx 
	+O(e^{-py}\|h\|_{H^1}).
\end{align}
Now we recall \eqref{eq6.4}. We have
\begin{align}
\label{eq4.16}
	- 2 \re \int \mathcal{R}_{y}Q hdx 
	&=
	-2 \int \mathcal{T}_{y}Q \mathcal{T}_{-y}Q dx
	\\ \notag
	&\quad +2 \rho  \re \int \mathcal{R}_{y}Q \mathcal{G}_{R,y}Q  dx 
	+\|\rho \mathcal{G}_{R,y}Q+h\|_{L^2}^2.
\end{align}
Combining \eqref{eq4.8.0}, \eqref{eq4.15} and \eqref{eq4.16},  we get
\begin{align*}
	\mu(t)
	&=2 \int \partial_x \mathcal{T}_{y}Q \partial_x \mathcal{T}_{-y}Q  dx 
	-2\rho \re \int \partial_x (\mathcal{R}_{y}Q ) \partial_x  (\mathcal{G}_{R,y}Q )dx
	\\
	&\quad +2 \int \mathcal{T}_{y}Q \mathcal{T}_{-y}Q dx
	-2 \rho \re \int \mathcal{R}_{y}Q  \mathcal{G}_{R,y}Q dx 
	\\
	&\quad +O(|\rho|^2 + \|h\|_{H^1}^2+e^{-py} \|h\|_{H^1}).
\end{align*}
where we used $\|\rho \mathcal{G}_{R,y}Q+h\|_{H^1}^2\lesssim |\rho|^2 + \|h\|_{H^1}^2$.
By \eqref{elliptic}, we have
\begin{align*}
	\int \partial_x \mathcal{T}_{y}Q \partial_x \mathcal{T}_{-y}Q  dx 
	+ \int \mathcal{T}_{y}Q \mathcal{T}_{-y}Q dx
	=\int (\mathcal{T}_{y}Q )^{p}\mathcal{T}_{-y}Q 
	\approx e^{-2y}.
\end{align*}
Thus we get
\begin{align*}
	&2\rho \left(\int \partial_x (\mathcal{R}_{y}Q ) \partial_x  (\mathcal{G}_{R,y}Q )dx + \re \int \mathcal{R}_{y}Q \mathcal{G}_{R,y}Q dx  \right)
	\\
	&\approx - \mu(t) + e^{-2y} + O(|\rho|^2 + \|h\|_{H^1}^2+e^{-py} \|h\|_{H^1}). 
\end{align*}
By integration by parts and \eqref{EEQ} again, we get
\begin{align*}
	&\int \partial_x (\mathcal{R}_{y}Q ) \partial_x  (\mathcal{G}_{R,y}Q )dx + \re \int \mathcal{R}_{y}Q \mathcal{G}_{R,y}Q dx
	\\
	&=\int \{(\mathcal{T}_{y}Q )^p -  (\mathcal{T}_{-y}Q )^p\}  \mathcal{G}_{R,y}Q dx 
	\\
	&=2 \int \chi_{R}^{+}(\mathcal{T}_{y}Q )^{p+1} dx + O(e^{-2 y})
	\\
	&=2 \int (\mathcal{T}_{y}Q )^{p+1} dx + O(e^{-2y}),
\end{align*}
where we used Lemma \ref{lem2.15}. 
Therefore we get
\begin{align*}
	|\rho| \approx  | \mu(t)| +e^{-2y} + O(|\rho|^2 + \|h\|_{H^1}^2+e^{-py} \|h\|_{H^1}).
\end{align*}
This means that
\begin{align*}
	|\rho| \approx  \mu(t) + O(e^{-2y} +\|h\|_{H^1}^2).
\end{align*}
This completes the proof.
\end{proof}

Next we consider the estimate of $\mu$ by using $S(u(t))=2S(Q)$. 
The following is a key in our proof. 

\begin{lemma}
\label{key}
Let $y$ be a positive number.
For sufficiently large $y>0$, it holds that
\begin{align*}
	S(\mathcal{R}_{y}Q) -2 S(Q)
	\approx e^{-2y}.
\end{align*}
\end{lemma}

\begin{proof}
By a direct calculation, we have
\begin{align*}
	S(\mathcal{R}_{y}Q) -2 S(Q)
	&=- \int \partial_x \mathcal{T}_{y}Q \partial_x \mathcal{T}_{-y}Q_ dx  
	- \int \mathcal{T}_{y}Q \mathcal{T}_{-y} Q dx 
	\\
	&\quad -\frac{1}{p+1}\left( \|\mathcal{R}_{y}Q\|_{L^{p+1}}^{p+1} -2 \|Q \|_{L^{p+1}}^{p+1}\right).
\end{align*}
By \eqref{elliptic}, we get
\begin{align*}
	- \int \partial_x \mathcal{T}_{y}Q \partial_x \mathcal{T}_{-y}Q dx 
	 - \int \mathcal{T}_{y}Q \mathcal{T}_{-y} Q dx
	&=-\int (\mathcal{T}_{y}Q)^{p} \mathcal{T}_{-y}Qdx
	\\
	&=-\int_{0}^{\infty} (\mathcal{T}_{y}Q)^{p} \mathcal{T}_{-y}Qdx 
	+O(e^{-py})
\end{align*}
where we used Lemma \ref{lem2.15}. 
Since $|\mathcal{R}_{y}Q |^{p+1}$ is even, it holds from Lemma \ref{lem2.15} that
\begin{align*}
	\|\mathcal{R}_{y}Q \|_{L^{p+1}}^{p+1} -2 \|Q \|_{L^{p+1}}^{p+1}
	=2\int_{0}^{\infty} (|\mathcal{R}_{y}Q |^{p+1} -|\mathcal{T}_{y}Q |^{p+1})dx + O(e^{-(p+1) y}).
\end{align*}
By the Taylor expansion, it holds that
\begin{align*}
	&\int_{0}^{\infty} (|\mathcal{R}_{y}Q |^{p+1} -|\mathcal{T}_{y}Q |^{p+1})dx
	\\
	&= -\int_{0}^{\infty} (p+1)(\mathcal{T}_{y}Q )^{p}\mathcal{T}_{-y}Q  + p(p+1)(\mathcal{T}_{y}Q )^{p-1}(\mathcal{T}_{-y}Q )^2  + O((\mathcal{T}_{-y}Q )^3)dx 
	\\
	&= -\int_{0}^{\infty} (p+1)(\mathcal{T}_{y}Q )^{p}\mathcal{T}_{-y}Q dx + O(e^{-3y}).
\end{align*}
Therefore, we obtain
\begin{align*}
	S (\mathcal{R}_{y}Q ) -2 S (Q )
	&=\int_{0}^{\infty} (\mathcal{T}_{y}Q )^{p} \mathcal{T}_{-y}Q dx +O(e^{-3 y}) 
	\\
	&\approx e^{-2y}+ O(e^{-3y}).
\end{align*}
For large $y$, we obtain the result. 
\end{proof}

Since $y(t) > R$, the assumption in the above lemma is satisfied by taking  $R$ large enough. 

\begin{lemma}
We have
\begin{align*}
	e^{-2y} + \|h\|_{H^1}^2 \lesssim |\rho|^2.
\end{align*}
\end{lemma}

\begin{proof}
We have
\begin{align*}
	0=S(u(t))-2S(Q)
	=S(\mathcal{R}_{y}Q + g)- S (\mathcal{R}_{y}Q) + S(\mathcal{R}_{y}Q )-2S(Q).
\end{align*}
By Lemma \ref{key}, the second term is estimated by 
\begin{align*}
	S(\mathcal{R}_{y}Q)-2S(Q) \gtrsim e^{-2y}.
\end{align*}
By Taylor expansion, we have
\begin{align*}
	S(\mathcal{R}_{y}Q + g)- S(\mathcal{R}_{y}Q) 
	=\langle S'(\mathcal{R}_{y}Q) , g \rangle
	+\frac{1}{2}\langle S''(\mathcal{R}_{y}Q) g, g \rangle
	+o(\|g\|_{H^1}^2),
\end{align*}
where we note that 
\begin{align*}
	S'(\mathcal{R}_{y}Q)
	&= -\partial_x^2 \mathcal{R}_{y}Q+ \mathcal{R}_{y}Q -|\mathcal{R}_{y}Q|^{p-1}\mathcal{R}_{y}Q,
	\\
	\langle S''(\mathcal{R}_{y}Q) \phi, \psi \rangle
	&=
	\re\int \nabla \phi \overline{\nabla \psi} + \omega \phi \overline{\psi} dx
	-\re \int |\mathcal{R}_{y}Q|^{p-1} (p\phi_1 + i \phi_2) \overline{\psi}dx
\end{align*}
for $\phi=\phi_1+i\phi_2, \psi\in H^{1}$.

First we give an estimate of $\langle S'(\mathcal{R}_{y}Q) , g \rangle$. By \eqref{EEQ}, we have
\begin{align*}
	\langle S'(\mathcal{R}_{y}Q) , g \rangle
	=\langle (\mathcal{T}_{y}Q)^{p} -(\mathcal{T}_{-y}Q)^{p}  -|\mathcal{R}_{y}Q|^{p-1}\mathcal{R}_{y}Q, g \rangle.
\end{align*}
By a nonlinear estimate, we obtain
\begin{align*}
	|(\mathcal{T}_{y}Q)^{p} -(\mathcal{T}_{-y}Q)^{p}  -|\mathcal{R}_{y}Q|^{p-1}\mathcal{R}_{y}Q|
	\lesssim  \{(\mathcal{T}_{y}Q)^{p-2}+ (\mathcal{T}_{y}Q)^{p-2}\} \mathcal{T}_{y}Q\mathcal{T}_{-y}Q.
\end{align*}
Thus we get
\begin{align}
\label{eq6.5}
	&\int |(\mathcal{T}_{y}Q)^{p} -(\mathcal{T}_{-y}Q)^{p}  -|\mathcal{R}_{y}Q|^{p-1}\mathcal{R}_{y}Q||g|dx
	\\ \notag
	&\lesssim \int \{(\mathcal{T}_{y}Q)^{p-2}+ (\mathcal{T}_{y}Q)^{p-2}\} \mathcal{T}_{y}Q\mathcal{T}_{-y}Q|g|dx
	\\ \notag
	&\lesssim e^{-2y}\|g\|_{L^2}.
\end{align}
By the Young inequality and $\|g\|_{L^2}^2 \lesssim |\rho|^2 + \|h\|_{H^1}^2$, we get
\begin{align*}
	|\langle S'(\mathcal{R}_{y}Q) , g \rangle|
	\lesssim e^{-2y}\|g\|_{L^2}
	\lesssim C_{\varepsilon}e^{-4y} + |\rho|^2 + \varepsilon \|h\|_{H^1}^2
\end{align*}
for arbitrary small $\varepsilon>0$. 
Next we consider $\langle S''(\mathcal{R}_{y}Q) g, g \rangle$. 
We set
\begin{align*}
	B(\phi, \psi):=\langle S''(\mathcal{R}_{y}Q) \phi, \psi \rangle.
\end{align*}
Since $B$ is bilinear and $g=\rho \mathcal{G}_{R,y}Q +h$, we have
\begin{align*}
	B(g,g)= |\rho|^2 B(\mathcal{G}_{R,y}Q,\mathcal{G}_{R,y}Q) + 2\rho B(\mathcal{G}_{R,y}Q,h)+B(h,h).
\end{align*}
It is obvious that
\begin{align*}
	|\rho|^2 |B(\mathcal{G}_{R,y}Q,\mathcal{G}_{R,y}Q)| \lesssim |\rho|^2.
\end{align*}
Now we have
\begin{align}
\label{eq4.18}
	B(\mathcal{G}_{R,y}Q,h)
	=B(\mathcal{R}_{y}Q,h) + B(\mathcal{G}_{R,y}Q-\mathcal{R}_{y}Q,h).
\end{align}
The second term of \eqref{eq4.18} is estimated by 
\begin{align}
\label{eq4.19}
	 |B(\mathcal{G}_{R,y}Q-\mathcal{R}_{y}Q,h)|
	 \lesssim e^{-y}\|h\|_{H^1} + R^{-1}\|h\|_{H^1}
\end{align}
since we have
\begin{align*}
	\int |\mathcal{G}_{R,y}Q -\mathcal{R}_{y}Q ||h|dx
	=\int  |\chi_{R}^{+}-1| \mathcal{T}_{y}Q|h| +  |1-\chi_{R}^{-}|\mathcal{T}_{-y}Q|h| dx
	\lesssim e^{-y}\|h\|_{L^2}
\end{align*}
and 
\begin{align*}
	&\int |\partial_x(\mathcal{G}_{R,y}Q -\mathcal{R}_{y}Q  )||\partial_x h|dx
	\\
	&=\int |\chi_{R}^{+}-1||\partial_x\mathcal{T}_{y}Q ||\partial_x h| +  |1-\chi_{R}^{-}||\partial_x \mathcal{T}_{-y}Q ||\partial_x h| dx 
	+O(R^{-1}\|h\|_{L^2})
	\\
	&\lesssim e^{-y}\|\partial_x h\|_{L^2}+O(R^{-1}\|h\|_{L^2}).
\end{align*}
The first term of \eqref{eq4.18}  is estimated as follows. By \eqref{EEQ} and the fact that $\mathcal{R}_{y}Q $ is real-valued, we get
\begin{align*}
	B(\mathcal{R}_{y}Q,h)
	&=\int_{\mathbb{R}} \{(\mathcal{T}_{y}Q)^{p} -(\mathcal{T}_{-y}Q )^{p}  -p|\mathcal{R}_{y}Q |^{p-1}\mathcal{R}_{y}Q \} h dx
	\\
	&=p\int_{\mathbb{R}} \{(\mathcal{T}_{y}Q )^{p} -(\mathcal{T}_{-y}Q )^{p}  -|\mathcal{R}_{y}Q |^{p-1}\mathcal{R}_{y}Q \} h dx
	\\
	& \quad +(1-p)\int_{\mathbb{R}} \{(\mathcal{T}_{y}Q )^{p} -(\mathcal{T}_{-y}Q )^{p}\} h dx.
\end{align*}
This first term is estimated by 
\begin{align*}
	\left|\int \{(\mathcal{T}_{y}Q )^{p} -(\mathcal{T}_{-y}Q )^{p}  -|\mathcal{R}_{y}Q |^{p-1}\mathcal{R}_{y}Q \} h dx \right|
	\lesssim e^{-2y} \|h\|_{H^1}
\end{align*}
as in \eqref{eq6.5}. We recall \eqref{eq4.11} and thus the second term is estimated by
\begin{align*}
	\left|\int  \{(\mathcal{T}_{y}Q )^{p} -(\mathcal{T}_{-y}Q )^{p}\} h dx \right|
	\lesssim e^{-py}\|h\|_{H^1}.
\end{align*}
Thus we have
\begin{align}
\label{eq4.20}
	|B(\mathcal{R}_{y}Q,h)| \lesssim e^{-2y} \|h\|_{H^1}.
\end{align}
Combining \eqref{eq4.18}, \eqref{eq4.19} and \eqref{eq4.20}, we obtain
\begin{align*}
	|\rho||B(\mathcal{G}_{R,y}Q,h)|
	&\lesssim e^{-y}|\rho|\|h\|_{H^1} + R^{-1}|\rho|\|h\|_{H^1}
	\\
	&\lesssim |\rho|^2 + e^{-2y}\|h\|_{H^1}^2+R^{-2}\|h\|_{H^1}^2
\end{align*}
by the Young inequality. 
At last, we calculate $B(h,h)$. We divide it into three terms as follows. 
We have
\begin{align*}
	B(h,h)=B(\chi_{R}h,\chi_{R}h)+2B(\chi_{R}h,\chi_{R}^{c}h)+B(\chi_{R}^{c}h,\chi_{R}^{c}h).
\end{align*}
We consider the middle term $B(\chi_{R}h,\chi_{R}^{c}h)$. We have
\begin{align*}
	B(\chi_{R}h,\chi_{R}^{c}h)
	&=\langle \partial_x (\chi_{R}h) , \partial_x(\chi_{R}^{c}h) \rangle
	+ \langle \chi_{R}h , \chi_{R}^{c}h \rangle
	\\
	&-p\langle |\mathcal{R}_{y}Q |^{p-1} \chi_{R}h_1 , \chi_{R}^{c}h_1 \rangle
	-\langle |\mathcal{R}_{y}Q |^{p-1} \chi_{R}h_2 , \chi_{R}^{c}h_2 \rangle.
\end{align*}
It holds that
\begin{align*}
	\langle \partial_x (\chi_{R}h) , \partial_x(\chi_{R}^{c}h) \rangle
	&= \int  \chi_{R} \chi_{R}^{c}|\partial_x h|^2 dx + O(R^{-1}\|h\|_{H^1}^{2})
\end{align*}
and 
\begin{align*}
	\langle  \chi_{R}h ,  \chi_{R}^{c}h \rangle
	=\int  \chi_{R} \chi_{R}^{c}| h|^2 dx.
\end{align*}
Moreover, by Lemma \ref{lem2.15}, we have
\begin{align*}
	&|p\langle |\mathcal{R}_{y}Q |^{p-1} \chi_{R}h_1 , \chi_{R}^{c}h_1 \rangle
	+\langle |\mathcal{R}_{y}Q |^{p-1} \chi_{R}h_2 , \chi_{R}^{c}h_2 \rangle|
	\\
	&\lesssim  \int |\mathcal{R}_{y}Q |^{p-1}  \chi_{R}\chi_{R}^{c} |h|^2dx
	\\
	&\lesssim   e^{-(p-1) y}\|h\|_{L^2}^2.
\end{align*}
Thus, we can estimate the middle term as follows. 
\begin{align*}
	B(\chi_{R}h,\chi_{R}^{c}h)
	\approx \int  \chi_{R} \chi_{R}^{c}(|h|^2+|\partial_x h|^2) dx +O(R^{-1}\|h\|_{H^1}^2+e^{-(p-1)y}\|h\|_{H^1}^2).
\end{align*}
In the similar way, we can estimate the third term $B(\chi_{R}^{c}h,\chi_{R}^{c}h)$ as follows. 
\begin{align*}
	B(\chi_{R}^{c}h,\chi_{R}^{c}h)
	\approx \int   (\chi_{R}^{c})^2 (|h|^2+|\partial_x h|^2) dx +O(R^{-1}\|h\|_{H^1}^2+e^{-(p-1)y}\|h\|_{H^1}^2).
\end{align*}
Next we consider the first term $B(\chi_{R}h,\chi_{R}h)$. 
By the symmetry, we have
\begin{align*}
	&B(\chi_{R}h,\chi_{R}h)
	\\
	&=2\left[ \int_{0}^{\infty} |\partial_x (\chi_{R}^{+}h)|^2 
	+ |\chi_{R}^{+}h|^2 dx
	-\int_{0}^{\infty} |\mathcal{R}_{y}Q |^{p-1}\{ p(\chi_{R}^{+}h_1)^2+(\chi_{R}^{+}h_2)^2\}dx\right]
	\\
	&=2\Phi(\mathcal{T}_{-y}(\chi_{R}^{+}h))
	+2\int_{0}^{\infty} (|\mathcal{T}_{y}Q |^{p-1} - |\mathcal{R}_{y}Q |^{p-1})\{ p(\chi_{R}h_1)^2+(\chi_{R}h_2)^2\}dx.
\end{align*}
Since it holds that 
\begin{align*}
	&\left| \int_{0}^{\infty} (|\mathcal{T}_{y}Q|^{p-1} - |\mathcal{R}_{y}Q|^{p-1})\{ p(\chi_{R}h_1)^2+(\chi_{R}h_2)^2\}dx \right|
	\\
	&\lesssim \int_{0}^{\infty} \{(\mathcal{T}_{y}Q)^{p-2}\mathcal{T}_{-y}Q+(\mathcal{T}_{-y}Q)^{p-1}\} |h|^2dx 
	\\
	&\lesssim e^{-2y}\|h\|_{H^1}^2.
\end{align*}
For $\Phi$, by Lemma \ref{coercivity}, we have
\begin{align*}
	 \Phi(\mathcal{T}_{-y}(\chi_{R}^{+}h)) \gtrsim  \|\chi_{R}h\|_{H^1}^2  -R^{-1}\|h\|_{H^1}^2.
\end{align*}
Therefore, combining all estimates above, we get
\begin{align*}
	e^{-2y} +  \|h\|_{H^1}^2  + 
	\lesssim |\rho|^2 + R^{-1}\|h\|_{H^1}^2 + e^{-2y}\|h\|_{H^1}^2
	+C_{\varepsilon}e^{-4y} + \varepsilon \|h\|_{H^1}^2.
\end{align*}
where we used 
\begin{align*}
	&\|\chi_{R}h\|_{H^1}^2  + \int  \chi_{R} \chi_{R}^{c}(|h|^2+|\partial_x h|^2) dx + \int  (\chi_{R}^{c})^2(|h|^2+|\partial_x h|^2) dx
	\\
	&\gtrsim \|h\|_{H^1}^2 - R^{-1}\|h\|_{H^1}^2.
\end{align*}
This implies 
\begin{align*}
	e^{-2y} +  \|h\|_{H^1}^2  
	\lesssim |\rho|^2
\end{align*}
by taking small $\varepsilon>0$ and sufficiently large $R$. 
\end{proof}

\begin{corollary}
\label{cor4.8}
We have 
\begin{align*}
	e^{-2y} +  \|h\|_{H^1}^2 
	\lesssim \mu(t)^2.
\end{align*}
\end{corollary}

\begin{proof}
By Lemma \ref{lem4.5}, we have
\begin{align*}
	|\rho|^2 \lesssim  \mu(t)^2 +e^{-4y} +\|h\|_{H^1}^4.
\end{align*}
This implies the result by taking sufficiently small $\mu_0$, which ensures the smallness of $\|h\|_{H^1}$. 
\end{proof}

\begin{corollary}
We have
\begin{align*}
	(1+y)e^{-2y} \lesssim \mu(t)^{2-\delta} 
	\lesssim \mu_0^{1-\delta} \mu(t),
\end{align*}
where $\delta>0$ is arbitrary small. 
\end{corollary}

\begin{proof}
It follows from Corollary \ref{cor4.8} and $\mu(t)<\mu_0$ that
\begin{align*}
	(1+y)e^{-2y} \lesssim e^{-(2-\delta)y}
	\lesssim \mu(t)^{2-\delta} \lesssim \mu_0^{1-\delta} \mu(t)
\end{align*}
for any small $\delta$. 
\end{proof}


\subsection{Estimate of the derivatives of the parameters}

The derivatives of the parameters are estimated as follows. 

\begin{lemma}
\label{lem4.10}
We have
\begin{align*}
	|\theta'(t)|+|\rho'(t)|+|y'(t)| \lesssim \mu(t). 
\end{align*}
\end{lemma}

\begin{proof}
Since we have
\begin{align*}
	h=e^{-i\theta(t) - it}u(t) -(\mathcal{R}_{y(t)}Q+ \rho(t)\mathcal{G}_{R,y(t)}Q ),
\end{align*}
we get
\begin{align*}
	h'
	&=-i(\theta'(t) +1)e^{-i\theta(t) - i t}u(t) + e^{-i\theta(t) - i t}u'(t) 
	\\
	&\quad-(-y'(t)\mathcal{R}_{y(t)}^{+}\partial_x Q + \rho'(t)\mathcal{G}_{R,y(t)}Q - \rho(t)y'(t)\mathcal{G}_{R,y(t)}^{+}\partial_x Q ),
\end{align*}
where we define
\begin{align*}
	\mathcal{R}_{y}^{+}f(x)&:=\mathcal{T}_{y}f(x) +\mathcal{T}_{-y}f(x),
	\\
	 \mathcal{G}_{R,y(t)}^{+}f(x)&:=\chi_{R}^{+}(x)\mathcal{T}_{y}f(x) 
	 +\chi_{R}^{-}(x)\mathcal{T}_{-y}f(x).
\end{align*}
From this, we have
\begin{align*}
	ih'+\partial_x^2 h
	&=\theta'(t)(\mathcal{R}_{y(t)}Q+g)
	 +h
	\\
	&\quad  - |\mathcal{R}_{y(t)}Q+g|^{p-1}(\mathcal{R}_{y(t)}Q+g) + |\mathcal{R}_{y(t)}Q|^{p}
	\\
	&\quad-|\mathcal{R}_{y(t)}Q |^{p} + |\mathcal{T}_{y}Q|^{p}-|\mathcal{T}_{-y}Q |^{p}
	\\
	&\quad-i(-y'(t)\mathcal{R}_{y(t)}^{+}\partial_x Q + \rho'(t)\mathcal{G}_{R,y(t)}Q - \rho(t)y'(t)\mathcal{G}_{R,y(t)}^{+}\partial_x Q )
	\\
	&\quad  -\rho(t)\{(\partial_x^2 \chi_{R}^{+})\mathcal{T}_{y}Q -(\partial_x^2 \chi_{R}^{-})\mathcal{T}_{-y}Q \}
	\\
	&\quad  -\rho(t)\{(\partial_x \chi_{R}^{+})\partial_x \mathcal{T}_{y}Q -(\partial_x \chi_{R}^{-})\partial_x \mathcal{T}_{-y}Q \}
	\\
	&\quad +\rho(t)\mathcal{G}_{R,y(t)}(Q^{p})
\end{align*}
where we use \eqref{elliptic} (or \eqref{EEQ}). 

Since we have
\begin{align*}
	 \im \int h' \chi_{R}^{+} \mathcal{T}_{y}Q dx =0
\end{align*}
by the orthogonality condition \eqref{ortho}, multiplying $\chi_{R}^{+} \mathcal{T}_{y}Q$ into this equation and taking the integral and the real part, we obtain
\begin{align*}
	\theta'(t)
	= O(\|h\|_{H^1}+\|g\|_{H^1}+|\rho| +e^{-2y}).
\end{align*}

Similarly, since we have
\begin{align*}
	 \re \int h' \chi_{R}^{+} (\mathcal{T}_{y}Q)^{p} dx =0
\end{align*}
by the orthogonality condition \eqref{ortho}, multiplying $\chi_{R}^{+} (\mathcal{T}_{y}Q_{\omega})^{p}$ into this equation and taking the integral and the imaginary part, we obtain
\begin{align*}
	\rho'(t)
	= O(\|h\|_{H^1}+|\theta'(t) |\|g\|_{H^1}+\|g\|_{H^1}+e^{-2y}+e^{-2y}|y'(t)|).
\end{align*}

In the same way, since we have
\begin{align*}
	 \re \int h' \partial_x(\chi_{R}^{+} \mathcal{T}_{y}Q ) dx =0
\end{align*}
by the orthogonality condition \eqref{ortho}, multiplying $\partial_x (\chi_{R}^{+} \mathcal{T}_{y}Q )$ into this equation and taking the integral and the imaginary part, we obtain
\begin{align*}
	y'(t)
	= O(\|h\|_{H^1}+|\theta'(t)|\|g\|_{H^1}+e^{-2 y}|\rho'(t)|+\|g\|_{H^1}+|\rho| +e^{-2 y}).
\end{align*}
These estimates imply the result.
\end{proof}

\section{Proof of scattering}

\subsection{Compactness of a critical element}

Suppose that Proposition \ref{prop2.13} (1) fails. 
Then there exists an odd solution $u$ with
\begin{align*}
	E(u)=2E(Q),\ 
	M(u)=2M(Q)
	\text{ and }
	K(u(t))> 0,
\end{align*}
where the solution is global by the variational argument, such that 
\begin{align*}
	\|u \|_{L_{t}^{a}L^{r}_{x}(\mathbb{R} \times\mathbb{R})}=\infty.
\end{align*}
We call the solution a critical element. We may consider only the positive time direction by time reversibility. Thus we may suppose that 
\begin{align*}
	\|u \|_{L_{t}^{a}L^{r}_{x}((0,\infty) \times\mathbb{R})}=\infty.
\end{align*}

\begin{proposition}[Compactness of a critical element]
\label{prop5.1}
Let $u$ be an odd solution with
\begin{align}
\label{eq5.1.0}
	E(u)=2E(Q),\ 
	M(u)=2M(Q)
	\text{ and }
	K(u(t))> 0,
\end{align}
such that 
\begin{align*}
	\|u \|_{L_{t}^{a}L^{r}_{x}((0, \infty)\times\mathbb{R})}=\infty.
\end{align*}
Then the solution $u$ satisfies the following compactness property: 
There exists a function $x:[0,\infty) \to [0,\infty)$ such that for any $\varepsilon>0$ there exists $R=R(\varepsilon)>0$ such that
\begin{align*}
	\int_{\{|x-x(t)|>R\}\cap \{|x+x(t)|>R\}} |\partial_x u(t,x)|^2 + |u(t,x)|^2 dx <\varepsilon
\end{align*}
for all $t \in [0,\infty)$.
\end{proposition}

\begin{proof}
Let $\{\tau_n\}$ be an arbitrary time sequence such that $\tau_n \to \infty$. The sequence $\{u(\tau_n)\}$ is bounded in $H^1(\mathbb{R})$ by  \eqref{eq5.1.0} and Proposition \ref{prop4.8}. By the linear profile decomposition (Lemma \ref{LPD}), we have, up to subsequence,
\begin{align*}
	u(\tau_n) = \sum_{j=1}^{J} e^{-i t_{n}^j \partial_x^2} \mathcal{R}_{x_n^j} \psi^j +R_n^J
\end{align*}
and the properties in the statement hold. We set $\psi_n^j := e^{-i t_{n}^j \partial_x^2} \mathcal{R}_{x_n^j} \psi^j$ for simplicity. 

It is easy to show that $J=0$ does not occur. Indeed, by the linear profile decomposition, if $J=0$, then $\|e^{it\partial_x^2 } u(\tau_n)\|_{L_{t}^{a}L_x^r((0,\infty)\times\mathbb{R})} \to 0$ as $n \to \infty$ (see Remark \ref{rmk1}). Then, by the long time perturbation (Lemma \ref{perturb}), we get $\|u\|_{L_t^aL_x^r((\tau_n,\infty)\times\R)} \lesssim 1$ for large $n \in \mathbb{N}$. This contradicts that $u$ does not scatter. 

Thus we have $J \geq 1$. We will prove that $J=1$ by contradiction. Suppose that $J\geq 2$ and thus we have at least two non-zero functions $\psi^1$ and $\psi^2$. 
First, by the linear profile decomposition, we have
\begin{align*}
	&\lim_{n \to \infty} \left(\sum_{j=1}^{J} M(\psi_n^j) + M(R_n^J)\right) = \lim_{n \to \infty} M(u(\tau_n)) =M(u_0) =2M(Q),
	\\
	&\lim_{n \to \infty} \left(\sum_{j=1}^{J} \|\partial_x \psi_n^j\|_{L^2}^2 + \|\partial_x R_n^J\|_{L^2}^2\right) = \lim_{n \to \infty} \|u(\tau_n)\|_{L^2}^2 \leq2 \|\partial_x Q\|_{L^2}^2.
\end{align*}
It holds that, for any $j$ and for large $n$,
\begin{align*}
	&\|\psi_n^j\|_{L^2}^{\sigma} \|\partial_x\psi_n^j\|_{L^2} 
	< 2^{p-1} \| Q\|_{L^2}^{\sigma}  \|\partial_x Q\|_{L^2},
	\\
	&\|R_n^J\|_{L^2}^{\sigma} \|\partial_xR_n^J\|_{L^2} 
	< 2^{p-1} \| Q\|_{L^2}^{\sigma}  \|\partial_x Q\|_{L^2}.
\end{align*}
From these and Lemma \ref{lem4.9}, we get $K(\psi_n^j),K(R_n^J) > 0$. It follows from $E\geq \frac{p-1}{2}K$ that $E(\psi_n^j),E(R_n^J)\geq 0$.
Now, since we also have 
\begin{align*}
	\lim_{n \to \infty} \left(\sum_{j=1}^{J} E(\psi_n^j) + E(R_n^J)\right) = \lim_{n \to \infty} E(u(\tau_n)) =E(u_0) =2E(Q),
\end{align*}
we obtain $E(\psi_n^j), E(R_n^J) \leq 2E(Q)$.

Since $J \geq 2$, there exists $\delta>0$ such that $M(\psi_n^j)^{\sigma}E(\psi_n^j) <2^{1+\sigma}M(Q)^{\sigma}E(Q) -\delta$ for all $j$. 
By reordering, we may choose $J_1, \cdots, J_4$ such that 
\begin{align}
\tag{Case 1}
	&1 \leq j \leq J_1 \Rightarrow t_n^j =0\ (\forall n \in \mathbb{N}) \text{ and } x_n^j =0 \ (\forall n \in \mathbb{N}),
	\\
\tag{Case 2}
	&J_1+1 \leq j \leq J_2 \Rightarrow t_n^j =0\ (\forall n \in \mathbb{N}) \text{ and } x_n^j \to \infty \ (n \to \infty),
	\\
\tag{Case 3}
	&J_2+1 \leq j \leq J_3 \Rightarrow |t_n^j|\to \infty \ (n \to \infty) \text{ and } x_n^j =0 \ (\forall n \in \mathbb{N}),
	\\
\tag{Case 4}
	&J_3+1 \leq j \leq J_4 \Rightarrow |t_n^j|\to \infty \ (n \to \infty) \text{ and } x_n^j \to \infty \ (n \to \infty),
\end{align}
where we are assuming that there is no $j$ such that $a \leq j \leq b$ if $a>b$. We will define nonlinear profiles associated with $\psi_n^j$. If there is no $j$ such that $J_k+1\leq j \leq J_{k+1}$ for some $k\in \{0,1,2,3\}$, where $J_0=0$, then skip the construction of nonlinear profiles in the following steps. 

{\bf Case 1.} 
We first consider the case of $1 \leq j \leq J_1$. By the orthogonality of the parameter $t_n^j$ and $x_n^j$, we note that $J_1=0$ or $1$. (Skip this step if $J_1=0$.) 
We define a solution $N$ to \eqref{NLS} with the initial data $N(0,x)=\mathcal{R}_{0}\psi^1(x)=\psi^1(x)-\psi^1(-x)$. Then, the solution $N$ is global and satisfies $\|N\|_{L_t^aL_x^r(\mathbb{R}\times \mathbb{R})} \lesssim 1$ by \cite{Inui17}, where the global dynamics for below the threshold for odd solutions is obtained, since $M(\mathcal{R}_{0}\psi^1)^{\sigma} E(\mathcal{R}_{0}\psi^1) <2^{1+\sigma}M(Q)^{\sigma} E(Q) -\delta$ and $K(\mathcal{R}_{0}\psi^1)\geq 0$. 

{\bf Case 2.} 
We consider the case of $J_1+1\leq j \leq J_2$. We define a solution $U^j$ to \eqref{NLS} with the initial data $\psi^j$. 
Now, we have
\begin{align*}
	\lim_{n \to \infty} M(\psi_n^j) =2M(\psi^j) \text{ and }
	\lim_{n \to \infty} E(\psi_n^j) =2E(\psi^j).
\end{align*}
Thus we get $M(\psi^j)^{\sigma} E(\psi^j) < M(Q)^{\sigma} E(Q) -\delta/2^{1+\sigma}$ and we also obtain $K(\psi^j)\geq 0$. 
It holds that $\|U^j\|_{L_t^aL_x^r(\mathbb{R}\times \mathbb{R})} \lesssim 1$ by \cite{FXC11} and \cite{AkNa13}. We set $U_n^j(t,x):=\mathcal{R}_{x_n^j}U^j(t)$. Then we also have $\|U_n^j\|_{L_t^aL_x^r(\mathbb{R}\times \mathbb{R})} \leq 2\|U^j\|_{L_t^aL_x^r(\mathbb{R}\times \mathbb{R})} \lesssim 1$.

{\bf Case 3.} 
We consider the case of $J_2+1 \leq j \leq J_3$. If $j$ satisfies $t_n^j \to -\infty$, then we define a solution $W^j$ to \eqref{NLS} that scatters to $\mathcal{R}_{0}\psi^j$ as $t \to + \infty$. Now, we have
\begin{align*}
	\| \psi_n^j \|_{L^{p+1}}^{p+1} = \|e^{-it_n^j\partial_x^2} \psi^j \|_{L^{p+1}}^{p+1} \to 0
\end{align*}
by the dispersive estimate. Since we have
\begin{align*}
	\lim_{n \to \infty} M(\psi_n^j)^{\sigma}E(\psi_n^j) 
	=M(\psi^j)^{\sigma} \frac{ \|\psi^j \|_{\dot{H}^1}^2 }{2}
	=M(W^j)^{\sigma}E(W^j)
\end{align*}
we obtain $M(W^j)^{\sigma}E(W^j) < 2^{1+\sigma}M(Q)^{\sigma} E(Q) -\delta$. We also have $K(W^j)\geq 0$ since $ \lim_{n \to \infty} (K(W^j(t_n^j)) - K(\psi_n^j))=0$ and $ \lim_{n \to \infty}K(\psi_n^j) \approx \|\psi^j\|_{\dot{H}^1}^2 >0$. Therefore, by \cite{Inui17}, we find that $W^j$ is global and satisfies $\|W^j \|_{L_t^aL_x^r(\mathbb{R}\times \mathbb{R})} \lesssim 1$.
If $j$ satisfies $t_n^j \to +\infty$, then we define a solution $W^j$ to \eqref{NLS} that scatters to $\mathcal{R}_{0}\psi^j$ as $t \to - \infty$. In the same way, we find that $W^j$ is global in both time directions and $\|W^j \|_{L_t^aL_x^r(\mathbb{R}\times \mathbb{R})} \lesssim 1$. We set $W_{n}^j (t,x):=W^j (t-t_n^j,x)$. 

{\bf Case 4.} 
We consider the case of $J_3 +1 \leq j \leq J_4$. If $j$ satisfies $t_n^j \to -\infty$, then we define a solution $V^j$ to \eqref{NLS} that scatters to $\psi^j$ as $t \to + \infty$. Now, we have
\begin{align*}
	\| \psi_n^j \|_{L^{p+1}}^{p+1} 
	= \|e^{-it_n^j\partial_x^2} \mathcal{R}_{x_n^j}\psi^j \|_{L^{p+1}}^{p+1}
	\leq 2 \|e^{-it_n^j\partial_x^2} \psi^j \|_{L^{p+1}}^{p+1} 
	\to 0
\end{align*}
by the dispersive estimate. Since we have
\begin{align*}
	\lim_{n \to \infty} M(\psi_n^j)^{\sigma}E(\psi_n^j) 
	=(2M(\psi^j))^{\sigma} \|\psi^j \|_{\dot{H}^1}
	=2^{1+\sigma}M(V^j)^{\sigma}E(V^j)
\end{align*}
we obtain $M(V^j)^{\sigma}E(V^j) < M(Q)^{\sigma} E(Q) -\delta/2^{1+\sigma}$. 
We also have $K(V^j)\geq 0$ since $ \lim_{n \to \infty} (K(V^j(t_n^j)) - K( e^{-it_n^j\partial_x^2} \psi^j))=0$ and $ \lim_{n \to \infty} K( e^{-it_n^j\partial_x^2} \psi^j)) \approx \|\psi^j\|_{\dot{H}^1}^2> 0$. Therefore, by \cite{FXC11} and \cite{AkNa13}, we find that $V^j$ is global and satisfies $\|V^j \|_{L_t^aL_x^r(\mathbb{R}\times \mathbb{R})} \lesssim 1$.
If $j$ satisfies $t_n^j \to +\infty$, then we define a solution $V^j$ to \eqref{NLS} that scatters to $\psi^j$ as $t \to - \infty$. Then, in the same way,  $V^j$ is global and $\|V^j \|_{L_t^aL_x^r(\mathbb{R}\times \mathbb{R})} \lesssim 1$. We set $V_{n}^j (t,x):=\mathcal{R}_{x_n^j}V^j (t-t_n^j,x)$. Then we have $\|V_n^j \|_{L_t^aL_x^r(\mathbb{R}\times \mathbb{R})} \leq 2\|V^j \|_{L_t^aL_x^r(\mathbb{R}\times \mathbb{R})} \lesssim 1$.

We denote all functions $N,U_n^j, W_n^j,V_n^j$ by $v_n^j$ and we define a nonlinear profile by $Z_n^J:=\sum_{j=1}^{J} v_n^j$.
We have
\begin{align*}
	Z_n^J=e^{it\partial_x^2 }u(\tau_n) +i\int_{0}^{t}  e^{i(t-s)\partial_x^2} |Z_n^J|^{p-1}Z_n^J ds -e^{it\partial_x^2} R_n^J+ s_n^J
\end{align*}
with $\|s_n^J\|_{L_t^aL_x^r} \to 0$ as $n \to \infty$ and $\limsup_{n \to \infty}\|e^{it\partial_x^2} R_n^J\|_{L_t^aL_x^r} < \varepsilon$ for large $J$. Moreover, $\limsup_{n \to \infty} \|Z_n^J\|_{L_t^aL_x^r}$ is bounded independently on $J$. By the long time perturbation, we obtain $\|u(\tau_n)\|_{L_t^aL_x^r} \lesssim  1$. This is a contradiction. Therefore, we get $J=1$. 
Thus, we get a parameter $(t_n,x_n)$ such that
\begin{align*}
	u(\tau_n) =e^{-it_n \partial_x^2} \mathcal{R}_{x_n} \psi +R_n
\end{align*}
and $\lim_{n \to \infty}\|R_n\|_{H^1} =0$. Assuming $|t_n| \to \infty$, we derive a contradiction to the non-scattering of $u$ by the same argument as in Case 3 and 4. Therefore, we have
\begin{align*}
	u(\tau_n) =\mathcal{R}_{x_n} \psi +R_n.
\end{align*}
By Proposition \ref{prop2.21} below, we get the statement. 
\end{proof}

\begin{proposition}
\label{prop2.21}
Let $u$ be an odd solution to \eqref{NLS} such that $M(u_0)=2M(Q)$, $E(u_0)=2E(Q)$ and $K(u_0)>0$. Assume that for any time sequence $\{t_n\} \subset [0,\infty)$ satisfying $t_n \to \infty$ as $n \to \infty$, there exist a subsequence of $\{t_n\}$, which is also denoted by $\{t_n\}$, $\{x_n\} \subset \mathbb{R}$ and $\psi \in H^1(\mathbb{R})\setminus\{0\}$ such that 
\begin{align*}
	\lim_{n \to \infty} \|u(t_n) - \mathcal{R}_{x_n}\psi \|_{H^1} =0.
\end{align*}
Then, there exists a function $x:[0,\infty) \to [0,\infty)$ such that for any $\varepsilon>0$ there exists $R=R(\varepsilon)>0$ such that
\begin{align*}
	\int_{\{|x-x(t)|>R\}\cap \{|x+x(t)|>R\}} |\partial_x u(t,x)|^2 + |u(t,x)|^2 dx <\varepsilon
\end{align*}
for all $t \in [0,\infty)$.
\end{proposition}



\begin{lemma}
\label{lemC.1}
Let $u(t)$ satisfy the assumption in Proposition \ref{prop2.21}. 
Then there exists $c>0$ such that
\begin{align*}
	\sup_{x_0 \geq 0 }\|u(t)\|_{L^2(\{|x-x_0|<1\}\cup\{|x+x_0|<1\})}\geq c
\end{align*}
for all $t \in \mathbb{R}$. 
\end{lemma}

\begin{proof}
Let $B_{y}:=\{|x-y|<1\}$ and $A_{y}:=B_{y}\cup B_{-y}$ for $y\geq 0$. 
By the symmetry, it is enough to show that there exist $c>0$ such that $\sup_{x_0 \in \mathbb{R} }\|u(t)\|_{L^2(A_{x_0})}\geq c$ for all $t$. 
If not, there exists $\{t_n\}$ satisfying $t_n \to \infty$ such that \begin{align*}
	\sup_{x_0 \in \mathbb{R} }\|u(t_n)\|_{L^2(A_{x_0})}\to 0.
\end{align*}
By the assumption, we get $\{x_n\}$ and $\psi$ such that $u(t_n) - \mathcal{R}_{x_n} \psi \to 0$ strongly in $H^1$. 
If $\{x_n\}$ is bounded, we may assume that $\{x_n\}$ converges to $x_{\infty} \in [0,\infty)$ by taking a subsequence and thus we may take $x_n \equiv 0$ by replacing $\mathcal{R}_{x_{\infty}}\psi$ by $\psi$. If $\{x_n\}$ is unbounded, we may assume that $x_n \to \infty$ as $n \to \infty$ by taking a subsequence. 
If $x_n\equiv 0$ (i.e. $u(t_n) \to \psi$ in $H^1$ strongly), then we have
\begin{align*}
	 \|\psi\|_{L^2(A_{x_0})}
	\leq \|u(t_n) - \psi\|_{L^2(A_{x_0})} 
	+\sup_{x_0\in \mathbb{R} }\|u(t_n)\|_{L^2(A_{x_0})}
\end{align*}
for any $x_0 \in \mathbb{R}$. The right hand side goes to zero. This means that $\psi= 0$. This contradicts $\psi\neq 0$. 

Next we consider the case of $x_n \to \infty$. 
Letting $\psi^\dagger(x)=\psi(-x)$, 
we have
\begin{align*}
	\|\mathcal{R}_{x_n}\psi\|_{L^2(A_{x_0+x_n})}^2
	&=\| \psi(x-x_n) - \psi^{\dagger}(x+x_n)\|_{L^2(A_{x_0+x_n})}^2
	\\
	&=\int_{A_{x_0+x_n}} |\psi(x-x_n)|^2 +|\psi^{\dagger}(x+x_n)|^2dx
	\\
	&\quad - \int_{A_{x_0+x_n}} 2\re \psi(x-x_n) \overline{\psi^{\dagger}(x+x_n)} dx
	\\
	&=\int_{\{|x-x_0|<1\}\cup\{|x+x_0+2x_n|<1\}} |\psi(x)|^2 dx
	\\
	&\quad  +\int_{\{|x-x_0-2x_n|<1\}\cup\{|x+x_0|<1\}} |\psi^{\dagger}(x)|^2 dx
	\\ 
	&\quad -2\re \int_{A_{x_0+x_n}} \psi(x-x_n)\overline{\psi^{\dagger}(x+x_n)} dx.
\end{align*}
When $x_n \to \infty$, we have
\begin{align*}
	& \int_{A_{x_0+x_n}} \psi(x-x_n)\overline{\psi^{\dagger}(x+x_n)} dx \to 0,
	 \\
	& \int_{\{|x+x_0+2x_n|<1\}} |\psi(x)|^2 dx + \int_{\{|x-x_0-2x_n|<1\}} |\psi^{\dagger}(x)|^2 dx \to 0. 
\end{align*}
Therefore, we obtain
\begin{align*}
	\|\mathcal{R}_{x_n}\psi\|_{L^2(A_{x_0+x_n})}^2 
	\to \|\psi\|_{L^2(B_{x_0})} + \|\psi^\dagger\|_{L^2(B_{-x_0})} 
	=2\|\psi\|_{L^2(B_{x_0})}
\end{align*}
On the other hand, we have
\begin{align*}
	\|\mathcal{R}_{x_n}\psi\|_{L^2(A_{x_0+x_n})} 
	&\leq \|u(t_n)-\mathcal{R}_{x_n}\psi\|_{L^2(A_{x_0+x_n})}
	 +\|u(t_n)\|_{L^2(A_{x_0+x_n})} 
	\\
	&\leq \|u(t_n)-\mathcal{R}_{x_n}\psi\|_{L^2(\mathbb{R})}
	 +\sup_{x_0 \in \mathbb{R}}\|u(t_n)\|_{L^2(A_{x_0})} 
\end{align*}
for any $n$. 
This implies that $\|\mathcal{R}_{x_n}\psi\|_{L^2(A_{x_0+x_n})}  \to 0$ as $n \to \infty$. These estimates mean that $\|\psi\|_{L^2(B_{x_0})}=0$ for any $x_0\in \mathbb{R}$ and thus $\psi(x)= 0$. 
This is a contradiction.  
\end{proof}

\begin{proof}[Proof of Proposition \ref{prop2.21}]

By Lemma \ref{lemC.1}, there exists a function $x: [0,\infty) \to [0,\infty)$ satisfying 
\begin{align}
\label{eqC.1}
	\|u(t)\|_{L^2(\{|x-x(t)|<1\}\cup\{|x+x(t)|<1\})}\geq c
\end{align}
for all $t \in \mathbb{R}$. Now we have the following. 

{\bf Claim.} $|x(t_n)-x_n|$ is bounded. 

\begin{proof}[Proof of Claim]
Since, for any sequence $\{t_n\}$, there exists $\{x_n\}$ and $\psi$ such that
\begin{align*}
	u(t_n) = \mathcal{R}_{x_n} \psi +o_n(1)
\end{align*}
by the assumption and 
\begin{align*}
	\|u(t_n)\|_{L^2(\{|x-x(t_n)|<1\}\cup\{|x+x(t_n)|<1\})}\geq c
\end{align*}
by \eqref{eqC.1},
we obtain
\begin{align*}
	 \|\mathcal{R}_{x_n}\psi\|_{L^2(\{|x-x(t_n)|<1\}\cup\{|x+x(t_n)|<1\})}\geq c/2
\end{align*}
for sufficiently large $n$. 
If $|x(t_n)-x_n|$ is unbounded, a simple calculation gives us $\|\mathcal{R}_{x_n}\psi\|_{L^2(\{|x-x(t_n)|<1\}\cup\{|x+x(t_n)|<1\})} \to 0$ as $n \to \infty$. This is a contradiction.
\end{proof}

We suppose that the result fails. 
Then there exists $\varepsilon>0$ and a time sequence $\{t_n\}$ such that
\begin{align*}
	\int_{\{|x-x(t_n)|>n\} \cap \{|x+x(t_n)|>n\}} |\partial_x u(t_n,x)|^2 + |u(t_n,x)|^2 dx \geq \varepsilon.
\end{align*}
By the assumption, there exist $\{x_n\}$ and $\psi$ such that
\begin{align*}
	u(t_n) = \mathcal{R}_{x_n} \psi +o_n(1).
\end{align*}
Thus we also have
\begin{align*}
	\int_{\{|x-x(t_n)|>n\} \cap \{|x+x(t_n)|>n\}} |\partial_x \mathcal{R}_{x_n} \psi(x)|^2 + |\mathcal{R}_{x_n} \psi(x)|^2 dx \geq \frac{\varepsilon}{2}
\end{align*}
for large $n$. 
This is a contradiction since $|x(t_n)-x_n|$ is bounded. 
\end{proof}

Next we will show that we can replace $x(t)$ in the compactness by the modulation parameter $y(t)$ for $t \in I_{\mu_0}$. 

\begin{lemma}
\label{lem5.2}
Let $\mu_0$ be sufficiently small. Then there exists $C>0$ such that
\begin{align*}
	|x(t)-y(t)|<C
\end{align*}
for $t \in I_{\mu_0}$.
\end{lemma}

\begin{proof}
If not, there exists a sequence $\{t_n\} \subset I_{\mu_0}$ such that 
\begin{align*}
	|x(t_n)-y(t_n)| \to \infty.
\end{align*}
Since $t_n \in I_{\mu_0}$, by the modulation, we write $u(t_n)=\mathcal{R}_{y(t_n)}Q + g$. Then
\begin{align*}
	c 
	&\leq \int_{\{|x-x(t_n)|<1\} \cup \{|x+x(t_n)|<1\}} |u(t_n,x)|^2 dx 
	\\
	&\lesssim \int_{\{|x-x(t_n)|<1\} \cup \{|x+x(t_n)|<1\}} |\mathcal{R}_{y(t_n)}Q|^2 dx  + \int_{\mathbb{R}} |g(t_n) |^2 dx.
\end{align*}
Since the first term goes to zero as $n \to \infty$,
we have
\begin{align*}
	\frac{c}{2} 
	\lesssim \|g(t_n) \|_{L^2}^2
\end{align*}
for sufficiently large $n$. On the other hand, we have $\|g(t_n) \|_{L^2}^2 \lesssim \mu(t_n) < \mu_0$. Thus, we get a contradiction since $\mu_0\ll 1$. 
\end{proof}

We define new function
\begin{align}
\label{eqX}
	X(t):=
	\begin{cases}
	x(t), &t\in [0,\infty)\setminus I_{\mu_0},
	\\
	y(t), &t\in I_{\mu_0}.
	\end{cases}
\end{align}
Then it is easy to check that
for any $\varepsilon>0$, there exists $R>0$ such that
\begin{align*}
	\int_{\{|x-X(t)|>R\} \cap \{|x+X(t)|>R\}} |\partial_x u(t,x)|^2 + |u(t,x)|^2 dx <\varepsilon.
\end{align*}

\begin{lemma}
\label{lem5.3.0}
Let $\{t_n\} \subset [0,\infty)$ be an arbitrary time sequence. 
If $|X(t_n)|$ is unbounded, taking a subsequence of $\{t_n\}$, which is also denoted by $\{t_n\}$, we have $\psi \in H^1$ such that
\begin{align*}
	u(t_n) - \mathcal{R}_{X(t_n)}\psi \to 0.
\end{align*}
If $|X(t_n)|$ is bounded, taking a subsequence of $\{t_n\}$, which is also denoted by $\{t_n\}$, we have $\psi \in H^1$ such that
\begin{align*}
	u(t_n) \to \psi.
\end{align*}
\end{lemma}

\begin{proof}
If $|X(t_n)|$ is unbounded, we may assume that $X(t_n) \to \infty$ by taking a subsequence. 
By Proposition \ref{prop5.1}, we have $\{x_n\}$ and $\psi \in H^1$ such that  
\begin{align*}
	u(t_n) - \mathcal{R}_{x_n}\psi  \to 0. 
\end{align*}
Since $|x(t_n) - X(t_n)|$ and $|x_n -x(t_n)|$ are bounded (see Lemma \ref{lem5.2} and the claim in the proof of Proposition \ref{prop2.21}), taking a subsequence, we may assume that $x(t_n) - X(t_n) \to x'$ and $x_n -x(t_n)\to x''$. 
Then we have
\begin{align*}
	u(t_n) - \mathcal{R}_{X(t_n)}\psi 
	&=u(t_n) - \mathcal{R}_{x_n}\psi  + \mathcal{R}_{x_n}\psi  - \mathcal{R}_{x(t_n)}\psi + \mathcal{R}_{x(t_n)}\psi  - \mathcal{R}_{X(t_n)}\psi 
	\\
	&\to \mathcal{R}_{x'}\psi + \mathcal{R}_{x''}\psi =: \tilde{\psi}.
\end{align*}
Then, if $\tilde{\psi}\neq 0$, by the long time perturbation, we get a contradiction to non-scattering of $u$. 

If $|X(t_n)|$ is bounded, we may assume that $X(t_n)$ converges to $X_{\infty}$ by taking a subsequence. 
By Proposition \ref{prop5.1}, we have $\{x_n\}$ and $\psi \in H^1$ such that  
\begin{align*}
	u(t_n) - \mathcal{R}_{x_n}\psi  \to 0. 
\end{align*}
Since $|x(t_n) - X(t_n)|$ is bounded, $|x(t_n)|$ is bounded and thus $\{x_n\}$ is bounded by boundedness of $|x_n -x(t_n)|$. Thus we can take $x_n =0$ and $u(t_n) \to \mathcal{R}_{0}\psi$. 
\end{proof}

\subsection{Extinction of the critical element}

We derive a contradiction. First we show that $X$ must be bounded. Second we will prove that the boundedness of $X$ implies a contradiction.


\subsubsection{$X(t)$ is bounded}

We show $X(t)$ is bounded by contradiction. 

\begin{proposition}
\label{prop5.3}
Let $u$ be an odd solution with $E(u)=2E(Q)$, $M(u)=2M(Q)$ and $K(u(t))\geq 0$ such that any $\varepsilon>0$ there exists $R=R(\varepsilon)>0$ such that
\begin{align*}
	\int_{\{|x-X(t)|>R\}\cap \{|x+X(t)|>R\}} |\partial_x u(t,x)|^2 + |u(t,x)|^2 dx <\varepsilon
\end{align*}
for all $t \in [0,\infty)$, where $X$ is defined in \eqref{eqX}.
Then $X(t)$ is bounded.
\end{proposition}

To prove this, we prepare some lemmas. 
%
%

\begin{lemma}
There exists $C_{\varepsilon}>0$  such that
\begin{align*}
	\int_{t_1}^{t_2} \mu(t)dt \lesssim C_\varepsilon(1 + \sup_{t \in [t_1,t_2]}|X(t)|) (\mu(t_1)+\mu(t_2))
\end{align*}
for any $t_1,t_2>0$
\end{lemma}

\begin{proof}
We give estimates of $J_R'$ and $A_R$. 

{\bf Claim 1.} We have $|J_R'(t)|\lesssim R \mu(t)$ for all $t>0$.

\begin{proof}[Proof of Claim 1]
When $\mu(t)>\mu_0$, we have
\begin{align*}
	|J_R'(t)| 
	\lesssim R  \int |u||\partial_x u|dx
	\lesssim R \|u\|_{L_t^{\infty}H^1}^2
	\lesssim R\frac{\mu(t)}{\mu_0}.
\end{align*}
When $\mu(t)\leq \mu_0$, since $\mathcal{R}_{y(t)}Q$ is real-valued, we have
\begin{align*}
	|J_R'(t)|
	&=\left|2R \im \int \{\overline{u} \partial_x u(t,x) -
	\overline{e^{i\theta(t)+i t}\mathcal{R}_{y(t)}Q }\partial_x (e^{i\theta(t)+i t}\mathcal{R}_{y(t)}Q )\} \partial_x \varphi\left( \frac{x}{R}\right) dx\right|
	\\
	&\lesssim R (\|u\|_{H^1}+\|\mathcal{R}_{y(t)}Q \|_{H^1})\|u-e^{i\theta(t)+i  t}\mathcal{R}_{y(t)}Q \|_{H^1}
	\\
	&\lesssim R \mu(t).
\end{align*}
Thus, in any cases, we have $|J_R'(t)|\lesssim R \mu(t)$. 
\end{proof}

{\bf Claim 2.} There exists sufficiently small $\tilde{\delta}>0$ such that $|A_{R}| \leq \tilde{\delta} \mu(t)$ for all $t>0$. 

\begin{proof}
By the compactness, for any $\varepsilon>0$ there exists $C_{\varepsilon}>0$ such that
\begin{align*}
	\int_{\{|x-X(t)|>C_{\varepsilon}\} \cap \{|x+X(t)|>C_{\varepsilon}\}} |\partial_x u(t,x)|^2 + |u(t,x)|^2 dx <\varepsilon.
\end{align*}
We choose 
\begin{align*}
	R= C_{\varepsilon} + \sup_{t \in[t_1,t_2]} |X(t)|
\end{align*}
so that it holds that
\begin{align*}
	\{|x|>R\} \subset \{|x-X(t)|>C_{\varepsilon}\} \cap \{|x+X(t)|>C_{\varepsilon}\}
\end{align*}
for $t \in [t_1,t_2]$. 
By the error estimate, we have
\begin{align*}
	|A_R(t)| \lesssim \varepsilon.
\end{align*}
If $\mu(t)>\mu_0$, we get
\begin{align*}
	|A_R(t)| \lesssim \varepsilon\frac{\mu(t)}{\mu_0}. 
\end{align*}

By Lemma \ref{lem2.16}, we get
\begin{align*}
	A_R(t)&=F_{R}(u(t)) - F_{\infty}(u(t)) 
	\\
	&\quad - F_{R}(e^{i\theta +i \omega t}\mathcal{R}_{y}Q_{\omega}) 
	+ F_{\infty}(e^{i\theta +i \omega t}\mathcal{R}_{y}Q_{\omega})+O((1+y)e^{-2\sqrt{\omega}y}).
\end{align*}
Now we have
\begin{align*}
	&|F_{R}(u(t)) - F_{\infty}(u(t))
	- F_{R}(e^{i\theta +i \omega t}\mathcal{R}_{y}Q_{\omega}) 
	+ F_{\infty}(e^{i\theta +i \omega t}\mathcal{R}_{y}Q_{\omega})|
	\\
	& \lesssim (\|u(t)\|_{H^1(|x|>R)} + \|\mathcal{R}_{y}Q_{\omega}\|_{H^1(|x|>R)} )
	\|u(t)-e^{i\theta +i \omega t}\mathcal{R}_{y}Q_{\omega}\|_{H^1}.
\end{align*}
By the compactness, $R=C_{\varepsilon}+\sup|X(t)|$ and $X(t)=y(t)$ for $t \in I_{\mu_0}$, we get
\begin{align*}
	(\|u(t)\|_{H^1(|x|>R)} + \|\mathcal{R}_{y}Q_{\omega}\|_{H^1(|x|>R)} )
	\|u(t)-e^{i\theta +i \omega t}\mathcal{R}_{y}Q_{\omega}\|_{H^1}
	\lesssim \varepsilon \mu(t).
\end{align*}
Thus we get
\begin{align*}
	|A_R(u(t))| \lesssim \varepsilon \mu(t) + \mu_0^{1-\delta} \mu(t). 
\end{align*}
In both cases, we have
\begin{align*}
	|A_R(u(t))| \lesssim \min\{\varepsilon+ \mu_0^{1-\delta}, \varepsilon \mu_0^{-1} \} \mu(t).
\end{align*}
Since $\mu_0$ is sufficiently small, taking $\varepsilon>0$ small, we get $|A_{R}(u(t))| \leq \tilde{\delta}\mu(t)$. 
\end{proof}

By Claim 2, we obtain
\begin{align*}
	J_R''(t) = 8K(u(t)) +A_R(u(t)) \gtrsim \mu(t) -\tilde{\delta} \mu(t)  \gtrsim \mu(t).
\end{align*}
Integrating this inequality on $[t_1,t_2]$, we get
\begin{align*}
	J_R'(t_2) - J_R'(t_1) \gtrsim \int_{t_1}^{t_2} \mu(t) dt.
\end{align*}
Since it follows from Claim 1 that
\begin{align*}
	|J_R'(t_2) - J_R'(t_1)|\lesssim R (\mu(t_1)+\mu(t_2)) = (C_\varepsilon + \sup_{t \in [t_1,t_2]}|X(t)|) (\mu(t_1)+\mu(t_2)),
\end{align*}
we get
\begin{align*}
	\int_{t_1}^{t_2} \mu(t)dt \lesssim (C_\varepsilon + \sup_{t \in [t_1,t_2]}|X(t)|) (\mu(t_1)+\mu(t_2)).
\end{align*}
This completes the proof. 
\end{proof}

\begin{lemma}
\label{lem5.6}
Let $\{t_n\}$ be a time sequence satisfying $t_n \to \infty$ as $n \to \infty$. Then we have
\begin{align*}
	X(t_n) \to \infty \Leftrightarrow \mu(t_n) \to 0.
\end{align*}
\end{lemma}

\begin{proof}
If $\mu(t_n) \to 0$, then we have $X(t_n)=y(t_n)$ for large $n$. Thus we have
\begin{align*}
	e^{-2y(t_n)} \lesssim \mu(t_n) \to 0
\end{align*}
and this implies $X(t_n)=y(t_n)\to \infty$. 

Next we consider the case of $X(t_n) \to \infty$. We use a contradiction argument and we suppose that there exists $\{t_n\}$ satisfying $t_n \to \infty$ and $\delta>0$ such that $X(t_n) \to \infty$ and $\mu(t_n) \geq \delta >0$. By the compactness, there exist a subsequence, which is also denoted by $\{t_n\}$, and $\psi \in H^1(\mathbb{R})$ such that 
\begin{align*}
	\|u(t_n) - \mathcal{R}_{X(t_n)} \psi \|_{H^1} \to 0
\end{align*}
as $n \to \infty$. Since $X(t_n) \to \infty$, we have
\begin{align*}
	&2M(Q )=M(u(t_n)) \to 2M(\psi),
	\\
	&2E(Q )=E(u(t_n)) \to 2E(\psi),
\end{align*}
and 
\begin{align*}
	-\delta \geq \|\partial_x u(t_n)\|_{L^2}^2 - 2\|\partial_x Q \|_{L^2}^2
	\to 2\|\partial_x \psi\|_{L^2}^2 -2\|\partial_x Q \|_{L^2}^2
\end{align*}
and thus we get
\begin{align*}
	M(Q)= M(\psi), 
	\quad
	E(Q)= E(\psi), 
	\quad
	\|\partial_x \psi\|_{L^2}^2 < \|\partial_x Q \|_{L^2}^2.
\end{align*}
By the result of \cite{CFR20}, the solution $v(t)$ to \eqref{NLS} with $v(0)=\psi$ scatters in positive or negative time direction. We define the approximate solution $\tilde{v}_n(t,x)=v(t,x-X(t_n)) - v(t,-x-X(t_n))$, which satisfies 
\begin{align*}
	i\partial_ t\tilde{v}_n +\partial_x^2 \tilde{v}_n +|\tilde{v}|^{p-1}\tilde{v}_n =e
\end{align*}
where $e=|\tilde{v}|^{p-1}\tilde{v}_n - (|\mathcal{T}_{X(t_n)}v|^{p-1}\mathcal{T}_{X(t_n)}v -|\mathcal{T}_{-X(t_n)}v|^{p-1}\mathcal{T}_{-X(t_n)}v)$. 
We note that $\|u(t_n) - \tilde{v}_n(0)\|_{H^1} \to 0$ as $n \to \infty$. 

In the case that $v$ scatters in positive time direction, we have $\|\tilde{v}_n\|_{L_{t}^{a}L_{x}^{r}(0,\infty)} \lesssim \|v\|_{L_{t}^{a}L_{x}^{r}(0,\infty)} <\infty$. 
By the long time perturbation, we get
\begin{align*}
	\|u\|_{L_{t}^{a}L_{x}^{r}(t_n,\infty)}
	= \|u(t_n+t)\|_{L_{t}^{a}L_{x}^{r}(0,\infty)} 
	\lesssim C + \|\tilde{v}_n\|_{L_{t}^{a}L_{x}^{r}(0,\infty)}
	\lesssim C 
\end{align*}
for large $n$. This contradicts non-scattering of $u$. 

In the case that $v$ scatters in negative direction, we have $\|\tilde{v}_n\|_{L_{t}^{a}L_{x}^{r}(-\infty,0)} \lesssim \|v\|_{L_{t}^{a}L_{x}^{r}(-\infty,0)} <\infty$. 
By the long time perturbation, we get
\begin{align*}
	\|u\|_{L_{t}^{a}L_{x}^{r}(-\infty,t_n)}
	= \|u(t_n+t)\|_{L_{t}^{a}L_{x}^{r}(-\infty,0)} 
	\lesssim C + \|\tilde{v}_n\|_{L_{t}^{a}L_{x}^{r}(-\infty,0)}
	\lesssim C 
\end{align*}
for large $n$. As $n \to \infty$, this also contradicts non-scattering of $u$. 
Therefore, we get $\mu(t_n) \to 0$. 
\end{proof}

\begin{lemma}
\label{lem5.7}
There exists  $C>0$ such that
\begin{align*}
	|X(t)-X(s)|\leq C
\end{align*}
for all $t,s \in [0,\infty)$ with $|t-s|\leq 1$.
\end{lemma}

\begin{proof}
We prove by contradiction that there exist $\eta>0$ and $C_0>0$ such that
\begin{align*}
	|X(t)-X(s)|\leq C_0
\end{align*}
for all $t,s \in [0,\infty)$ with $|t-s|\leq \eta$. We suppose that for any $\eta>0$, there exist $t_n$ and $s_n$ such that
\begin{align*}
	|t_n-s_n| \leq \eta \text{ and } |X(t_n) - X(s_n)| \to \infty
\end{align*}
as $n \to \infty$. 

It follows from the Duhamel formula and the Sobolev embedding that
\begin{align*}
	\|u(t) - u(s)\|_{L^2} \leq \int_{s}^{t} \| e^{i(t-\tau)\partial_x^2}|u|^{p-1}u\|_{L^2} d\tau \leq |t-s| \|u\|_{L_t^{\infty}H^1}^{p}    \lesssim  |t-s|
\end{align*}
for $s<t$. By the compactness, for any $\varepsilon>0$, there exists $R>0$ such that 
\begin{align*}
	\sup_{t\geq 0} \int_{\{|x-X(t)|>R\}\cap \{|x+X(t)|>R\}} |u(t,x)|^2 dx <\varepsilon.
\end{align*}
Let $B(t):=\{|x-X(t)|\leq R\}\cup \{|x+X(t)|\leq R\}$. Then we have
\begin{align}
\label{eq5.3}
	\| \1_{B(t)}u(t) - \1_{B(s)}u(s)\|_{L^2}
	&\leq \| \1_{B(t)}u(t) - u(t)\|_{L^2} 
	\\ \notag
	&\quad + \| u(t) - u(s)\|_{L^2} 
	\\ \notag
	&\quad +\| u(s) - \1_{B(s)}u(s)\|_{L^2}
	\\ \notag
	&\lesssim \varepsilon +|t-s|
\end{align}
for all $s<t$. 
On the other hand, since $|X(t_n) - X(s_n)| \to \infty$ and thus $B(t_n) \cap B(s_n) = \emptyset$ for large $n$, we have
\begin{align}
\label{eq5.4}
	\| \1_{B(t_n)}u(t_n) - \1_{B(s_n)}u(s_n)\|_{L^2}
	&=\| \1_{B(t_n)}u(t_n) \|_{L^2} + \|\1_{B(s_n)}u(s_n)\|_{L^2}
	\\ \notag
	&\geq 2 (M(u) - \varepsilon) = 2 (2M(Q) - \varepsilon)
\end{align}
for large $n$. Taking $\varepsilon>0$ and $\eta>0$ sufficiently small, we get a contradiction combining \eqref{eq5.3} and \eqref{eq5.4}. When $\eta \geq 1$, the statement holds obviously. If $\eta<1$, then for $s<t$ with $|t-s|<1$, take $s=t_0<t_1<\cdots < t_l = t$ such that $|t_j-t_{j+1}|<\eta$ for $j=0,\cdots,l-1$. Then $l \leq \eta^{-1}$ and thus we get
\begin{align*}
	|X(t)-X(s)| \leq \sum_{j=0}^{l-1} |X(t_j) - X(t_{j+1})| \leq l C_0 \leq C_0 \eta^{-1}.
\end{align*}
In any way, we obtain the statement by setting $C:=\max\{C_0,C_0 \eta^{-1}\}$. 
\end{proof}

\begin{lemma}
\label{lem5.8}
There exists a constant $C$ such that
\begin{align*}
	|X(\tau_1) - X(\tau_2)| \leq C \int_{\tau_1}^{\tau_2} \mu(t) dt
\end{align*}
for any $\tau_1,\tau_2$ satisfying $\tau_1+1\leq \tau_2$. 
\end{lemma}

\begin{proof}
We start with the following claim. 

{\bf Claim 1.} There exists $\mu_1$ such that 
\begin{align*}
	\inf_{t \in[\tau_1,\tau_1+2]} \mu(t) \geq \mu_1 
	\text{ or }
	\sup_{t \in[\tau_1,\tau_1+2]} \mu(t) < \mu_0
\end{align*}
for any $\tau_1>0$. 
\begin{proof}[Proof of Claim 1]
If not, for any $n$ there exists $t_n$ such that
\begin{align*}
	\inf_{t \in[t_n,t_n+2]} \mu(t) < n^{-1} 
	\text{ and }
	\sup_{t \in[t_n,t_n+2]} \mu(t) \geq \mu_0 
\end{align*}
and thus there exists $t_n',t_n'' \in [t_n,t_n+2]$ such that 
\begin{align*}
	\mu(t_n') < n^{-1} 
	\text{ and }
	\mu(t_n'') \geq \mu_0 
\end{align*}
Thus we have
\begin{align*}
	\mu(t_n') \to 0, 
	\quad
	\mu(t_n'') \geq \mu_0
	\text{ and }
	|t_n'-t_n''|\leq 2.
\end{align*}
We may assume that $t_n'-t_n'' \to \tau \in [-2,2]$ as $n \to \infty$ extracting a subsequence if necessary.  

Then we may assume that $X(t_n'')$ converges. Indeed, if $\{t_n''\}$ is bounded, $X(t_n'')$ is bounded by Lemma \ref{lem5.7}. When $X(t_n'')$ is bounded, we may $X(t_n'')$ converges by taking a subsequence. 
If $t_n''$ is unbounded, we may assume that $t_n'' \to \infty$ along a subsequence. If $X(t_n'')\to \infty$, then we must have $\mu(t_n'') \to 0$ by Lemma \ref{lem5.6}. However, this contradicts the definition of $t_n''$. 

By Lemma \ref{lem5.3.0}, we have
\begin{align*}
	u(t_n'') \to \psi \text{ strongly in } H^1.
\end{align*}
Therefore we have
\begin{align*}
	M(\psi) = 2M(Q_{\omega}),
	\quad
	E(\psi) = 2E(Q_{\omega}),
	\quad
	\|\partial_x (\psi) \|_{L^2}^2 < 2\|\partial_x Q_{\omega}\|_{L^2}^2 - \mu_0
\end{align*}
Then the solution $v$ to \eqref{NLS} with $v(0)=\psi$ is global and $\tilde{\mu}(v(t))<0$ for all $t$. By the continuity of the flow, we have
\begin{align*}
	\|u(t_n'' + \tau) - v(\tau)\|_{H^1} \to 0
\end{align*}
since the initial data satisfy $\|u(t_n'') - v(0)\|_{H^1} \to 0$. Thus we also have $\mu(u(t_n'' + \tau)) \to \mu(v(\tau)) > 0$ as $n \to \infty$ and thus $\mu(u(t_n'' + \tau)) > c_0 >0$ for large $n$ and some constant $c_0$. 

On the other hand, we have 
\begin{align}
\label{eq7.1}
	\| u(t_n') - u(t_n''+\tau)\|_{\dot{H}^1} \to 0
\end{align}
which is proved later. Then $\mu(u(t_n''+\tau)) \to 0$ since $\mu(u(t_n'))\to 0$ and $\| u(t_n') - u(t_n''+\tau)\|_{\dot{H}^1} \to 0$. This is a contradiction. 
We will prove \eqref{eq7.1}. Let $I_n=[t_n',t_n''+\tau]$ (or $I_n=[t_n''+\tau,t_n']$). Now we have
\begin{align*}
	\|u(t_n') - u(t_n'' + \tau)\|_{\dot{H}^1} 
	&=\|  \int_{I_n} e^{i(t-s)\partial_x^2} |u(s)|^{p-1}u(s)ds \|_{\dot{H}^1}
	\\
	&\lesssim \int_{I_n} \||u(s)|^{p-1}u(s)\|_{\dot{H}^{1}}ds
	\\
	&\lesssim \int_{I_n} \|u(s)\|_{L^{2(p-1)}}^{p-1}\|u(s)\|_{\dot{H}^{1}}ds
	\\
	&\lesssim |I_n| \|u\|_{L_t^{\infty}H^1}^{p} \to 0
\end{align*}
since $|I_n| \to 0$ by $t_n' - (t_n'' + \tau) \to 0$. 
\end{proof}

{\bf Claim 2.}
There exists a constant $C$ such that
\begin{align*}
	|X(\tau_1) - X(\tau_2)| \leq C \int_{\tau_1}^{\tau_2} \mu(t) dt
\end{align*}
for any $\tau_1,\tau_2$ satisfying $\tau_1+1 \leq \tau_2 \leq \tau_1+2$. 

\begin{proof}[Proof of Claim 2]
In the first case of Claim 1, we have
\begin{align*}
	\int_{\tau_1}^{\tau_2} \mu(t) dt \geq 2 \inf_{t \in[\tau_1,\tau_1+2]} \mu(t) \geq \mu_1
\end{align*}
for $\tau_1+1< \tau_2< \tau_1+2$. By Lemma \ref{lem5.7}, we have
\begin{align*}
	|X(\tau_1) -X(\tau_2)| 
	&\leq |X(\tau_1) -X(\tau_1+1)| +|X(\tau_1+1) -X(\tau_2)| 
	\\
	&\leq 2C 
	\lesssim \frac{2C}{\mu_1} \int_{\tau_1}^{\tau_2} \mu(t) dt. 
\end{align*}

In the second case of Claim 1, we have $[\tau_1,\tau_1+2] \subset I_{\mu_0}$ and thus $X(t)=y(t)$ for $t \in [\tau_1,\tau_1+2]$. We have
\begin{align*}
	|X'(t)|=|y'(t)| \lesssim  \mu(t)
\end{align*}
for $t \in (\tau_1,\tau_1+2)$. Integrating this on $[\tau_1,\tau_2]$, we get
\begin{align*}
	|X(\tau_2) - X(\tau_1)| 
	=\left| \int_{\tau_1}^{\tau_2}X'(t) dt \right| 
	\lesssim  \int_{\tau_1}^{\tau_2}\mu(t)dt .
\end{align*}
We get Claim 2. 
\end{proof}
At last, we show the statement. For arbitrary $\tau_1, \tau_2$ satisfying $\tau_1+1 < \tau_2$, we divide the interval $[\tau_1,\tau_2]$ into the intervals $[\sigma_0,\sigma_1]=[\tau_1,\sigma_1], [\sigma_1,\sigma_2], \cdots [\sigma_{k-1},\tau_2]=[\sigma_{k-1},\sigma_k]$, whose length are larger than 1 and less than 2. By Claim 2, we have
\begin{align*}
	|X(\sigma_j) - X(\sigma_{j+1})| \leq C \int_{\sigma_j}^{\sigma_{j+1}} \mu(t)dt.
\end{align*}
Therefore we obtain
\begin{align*}
	|X(\tau_2) - X(\tau_1)|
	&=|X(\sigma_0) - X(\sigma_{k})|
	\\
	&\leq \sum_{j=0}^{k-1} |X(\sigma_j) - X(\sigma_{j+1})|
	\\
	&\leq C\sum_{j=0}^{k-1} \int_{\sigma_j}^{\sigma_{j+1}} \mu(t)dt
	\\
	&=C \int_{\tau_1}^{\tau_2} \mu(t)dt.
\end{align*}
This completes the proof of Lemma \ref{lem5.8}.
\end{proof}

\begin{proof}[Proof of Proposition \ref{prop5.3}]
We prove the statement by contradiction.  
Suppose that $|X(t)|$ is unbounded. Then we take a sequence $\{t_n\}$ such that $|X(t_n)| \to \infty$ and $\sup_{t\in [0,t_n]}|X(t)|=|X(t_n)|$. By Lemma \ref{lem5.6}, we have $\mu(t_n) \to 0$. Let $N$ such that $C\mu(t_n) < 100^{-1}$ for all $n >N$. By Lemma and , for $n\gg N$, which ensures $t_n > t_N +1$ we have
\begin{align*}
	X(t_n) - X(t_N)
	&\leq C \int_{t_N}^{t_n} \mu(t) dt
	\\
	&\leq C (1+|X(t_n)|) (\mu(t_N)+\mu (t_n))
	\\
	&\leq \frac{2}{100}(1+X(t_n))
\end{align*}
and thus 
\begin{align*}
	X(t_n) \lesssim 1+ X(t_N) <\infty. 
\end{align*}
This is a contradiction.
\end{proof}

\subsubsection{If $X(t)$ is bounded, contradiction}

By the previous section, we see that $X(t)$ is bounded. In this section, we get a contradiction from the boundedness.

Since $X(t)$ is bounded, for any $\varepsilon>0$, there exists $R>0$ such that
\begin{align*}
	\int_{\{|x|>R\}} |\partial_x u(t,x)|^2 + |u(t,x)|^2 dx <\varepsilon.
\end{align*}

\begin{lemma}
It holds that
\begin{align*}
	\lim_{T\to \infty} \frac{1}{T} \int_{0}^{T} \mu(t) dt=0.
\end{align*}
\end{lemma}

\begin{proof}
We consider the localized virial identity $J_R$. 
By a direct calculation, we have $|J_R'(t)|\lesssim  R$. 
We also have $J_R''(t)=8 \mu(t) + A_R(u(t))$.  
We can estimate the error term by compactness as follows: 
\begin{align*}
	|A_R(u(t))| \lesssim \int_{|x|>R} |\partial_x u|^2 + |u|^{p+1} + \frac{1}{R^2}|u|^2 dx \leq \varepsilon
\end{align*}
for large $R>0$. We have
\begin{align*}
	\int_{0}^{T} 8 \mu(t) + A_R(u(t))dt =\int_{0}^{T}J_R''(t) dt \leq |J_R'(t)|+|J_R'(0)| \lesssim R.
\end{align*}
Therefore we have
\begin{align*}
	\int_{0}^{T} \mu(t) dt \lesssim R +\varepsilon T
\end{align*}
and thus
we get
\begin{align*}
	\frac{1}{T}\int_{0}^{T} \mu(t) dt  \to 0
\end{align*}
as $T \to \infty$. 
\end{proof}

\begin{corollary}
\label{cor5.9}
There exists a time sequence $\{t_n\}$ such that 
\begin{align*}
	\lim_{n \to \infty} \mu(t_n) =0. 
\end{align*}
\end{corollary}

\begin{proof}
If not, there exists $\delta>0$ such that $\mu(t)>\delta$ for all $t$. This implies
\begin{align*}
	\frac{1}{T} \int_{0}^{T} \mu(t) dt\geq \delta.
\end{align*}
This is a contradiction. 
\end{proof}

Since we have a time sequence $\{t_n\}$ such that $\lim_{n \to \infty} \mu(t_n) =0$ by Corollary \ref{cor5.9}, Lemma \ref{lem5.6} implies that $X(t_n)$ must tend to $\infty$. This contradicts that $X$ is bounded. 
As a conclusion, we get Proposition \ref{prop2.13} (1).


\section{Proof of blow-up}

In this section, we prove the blow-up result, Theorem \ref{thm1.2} (2) through Theorem \ref{thm4.7} (2).

\begin{lemma}
\label{lem2.4}
Let $\varphi \in C^1(\mathbb{R};\mathbb{R})$ be even and $f \in H_{\odd}^1(\mathbb{R})$. Assume that they satisfy $\int_{\mathbb{R}} |\partial_x \varphi |^2 |f|dx < \infty$, $M(f)=2M(Q)$, and $E(f)=2E(Q)$. Then we have the following inequality.
\begin{align*}
	\left| \Im \int_{\mathbb{R}} \partial_x \varphi(x) \partial_x f(x) \overline{f(x)} dx\right| \lesssim \mu(f)^2 \int_{\mathbb{R}} |\partial_x \varphi(x)|^2 |f(x)|^2 dx. 
\end{align*}
\end{lemma}

\begin{proof}
Let $\lambda \in \mathbb{R}$. 
By the Gagliardo-Nirenberg ineqaulity (Lemma \ref{GN}), we have
\begin{align*}
	\|f\|_{L^{p+1}}^{p+1} 
	=\|e^{i\lambda \varphi} f\|_{L^{p+1}}^{p+1}
	\leq C_{GN}^{\odd} \| f\|_{L^2}^{\frac{p+3}{2}}( \| (e^{i\lambda \varphi} f)'\|_{L^2}^2 )^{\frac{p-1}{4}}.
\end{align*}
Now, we have
\begin{align*}
	\| (e^{i\lambda \varphi} f)'\|_{L^2}^2 
	&=\lambda^2 \|\varphi' f\|_{L^2}^2 +2 \lambda \Im \int_{\mathbb{R}} \varphi'f' \overline{f} dx+ \|f'\|_{L^2}^2.
\end{align*}
Therefore, we obtain 
\begin{align*}
	\|f\|_{L^{p+1}}^{p+1} 
	\leq	C_{GN}^{\odd} \| f\|_{L^2}^{\frac{p+3}{2}}\left( \lambda^2 \|\varphi' f\|_{L^2}^2 +2 \lambda \Im \int_{\mathbb{R}} \varphi'f' \overline{f} dx+ \|f'\|_{L^2}^2 \right)^{\frac{p-1}{4}}.
\end{align*}
This means that
\begin{align*}
	\lambda^2 \|\varphi' f\|_{L^2}^2 +2 \lambda \Im \int_{\mathbb{R}} \varphi'f' \overline{f} dx+ \|f'\|_{L^2}^2 -\left( \frac{\|f\|_{L^{p+1}}^{p+1} }{C_{GN}^{\odd} \| f\|_{L^2}^{\frac{p+3}{2}}}\right)^{\frac{4}{p-1}} \geq 0
\end{align*}
This is the inequality related to the quadratic equation for $\lambda$. Thus, we obtain
\begin{align*}
	\left|  \Im \int_{\mathbb{R}} \varphi'f' \overline{f} dx \right|^2 
	\leq  \|\varphi' f\|_{L^2}^2\left\{ \|f'\|_{L^2}^2 -\left( \frac{\|f\|_{L^{p+1}}^{p+1} }{C_{GN}^{\odd} \| f\|_{L^2}^{\frac{p+3}{2}}}\right)^{\frac{4}{p-1}}\right\}.
\end{align*}
Now, we have $ \|f'\|_{L^2}^2  = 2\|\partial_x Q\|_{L^2}^2 - \mu(f)$ and it holds from $E(f)=2E(Q)$ that 
$\|f\|_{L^{p+1}}^{p+1} =2\|Q\|_{L^{p+1}}^{p+1} - \frac{p+1}{2} \mu(f) $
Therefore, we obtain
\begin{align*}
	&\|f'\|_{L^2}^2 -\left( \frac{\|f\|_{L^{p+1}}^{p+1} }{C_{GN}^{\odd} \| f\|_{L^2}^{\frac{p+3}{2}}}\right)^{\frac{4}{p-1}}
	\\
	&=2\|\partial_x Q \|_{L^2}^2 - \mu(f)  -\left( \frac{\|f\|_{L^{p+1}}^{p+1} }{C_{GN}^{\odd} \| f\|_{L^2}^{\frac{p+3}{2}}}\right)^{\frac{4}{p-1}}
	\\
	&= 2\|\partial_x Q \|_{L^2}^2 - \mu(f) - \left(2\|Q \|_{L^{p+1}}^{p+1} - \frac{p+1}{2} \mu(f)\right)^{\frac{4}{p-1}}(C_{GN} 2^{-\frac{p-5}{4}}\| Q \|_{L^2}^{\frac{p+3}{2}})^{-\frac{4}{p-1}},
\end{align*}
where we also used $\|f\|_{L^2}^2 = 2\|Q\|_{L^2}^2$ and $C_{GN}^{\odd} =2^{-\frac{p-1}{2}} C_{GN}$. 
Then, by the Taylor expansion and the Pohozaev identity
\begin{align*}
	(C_{GN} \| Q \|_{L^2}^{\frac{p+3}{2}})^{-\frac{4}{p-1}}
	= \frac{p-1}{2(p+1)} \|Q \|_{L^{p+1}}^{\frac{(p-5)(p+1)}{p-1}},
\end{align*}
we obtain 
\begin{align*}
	 &2\|\partial_x Q \|_{L^2}^2 - \mu (f) - \left(2\|Q \|_{L^{p+1}}^{p+1} - \frac{p+1}{2} \mu(f)\right)^{\frac{4}{p-1}}(C_{GN}2^{-\frac{p-5}{4}} \| Q \|_{L^2}^{\frac{p+3}{2}})^{-\frac{4}{p-1}}
	 \\
	 & \leq 2\|\partial_x Q \|_{L^2}^2 - \mu(f) 
	 \\
	 &\quad -\left(2^{\frac{4}{p-1}} \|Q \|_{L^{p+1}} ^{\frac{4(p+1)}{p-1}} - \frac{2(p+1)}{p-1} 2^{-\frac{p-5}{p-1}} \|Q \|_{L^{p+1}}^{-\frac{(p-5)(p+1)}{p-1}} \mu (f)-C| \mu(f)|^2\right)
	 \\
	 &\quad \quad \quad \times (C_{GN}2^{-\frac{p-5}{4}} \| Q \|_{L^2}^{\frac{p+3}{2}})^{-\frac{4}{p-1}}
	 \\
	 &= C'| \mu(f)|^2. 
\end{align*}
Combining these estimates, we get
\begin{align*}
	\left|  \Im \int_{\mathbb{R}} \varphi'f' \overline{f} dx \right|^2 
	&\leq  \|\varphi' f\|_{L^2}^2\left\{ \|f\|_{H}^2 -\left( \frac{\|f\|_{L^{p+1}}^{p+1} }{C_{GN}^{odd} \| f\|_{L^2}^{\frac{p+3}{2}}}\right)^{\frac{4}{p-1}}\right\}
	\\
	&\lesssim  \|\varphi' f\|_{L^2}^2 |\mu(f)|^2.
\end{align*}
This completes the proof.
\end{proof}

\begin{corollary}
\label{cor2.5}
Under the assumption of Lemma \ref{lem2.4} with $K(f)<0$, we have
\begin{align*}
	\left| \Im \int_{\mathbb{R}} \partial_x \varphi(x) \partial_x f(x) \overline{f(x)} dx\right| \lesssim |K(f)|^2 \int_{\mathbb{R}} |\partial_x \varphi(x)|^2 |f(x)|^2 dx. 
\end{align*}
\end{corollary}

\begin{proof}
This follows from Lemmas \ref{lem2.11} and \ref{lem2.4}. 
\end{proof}

\begin{proposition}
\label{prop2.7}
We assume the same assumption in Proposition \ref{prop2.13} (2). 
Suppose that $u$ is global in positive time direction. 
Then we have
\begin{align*}
	\im \int_{\mathbb{R}} x \partial_x u(t,x) \overline{u(t,x)} dx >0
\end{align*}
for all existence time $t$.  
Moreover, there exists $c>0$ such that
\begin{align*}
	\int_{t}^{\infty} |\mu(s)|ds \lesssim e^{-ct}
\end{align*}
for any $t>0$. 
\end{proposition}

\begin{proof}
Recall $J(t)=\int_{\mathbb{R}} x^2|u(t,x)|^2 dx$. 
We will first show that $J'(t)>0$ for all existence time. If not, there exists $t_1$ such that $J'(t_1) \leq 0$. Since $y''<0$ for all existence time, we have
\begin{align*}
	J'(t_2)- J'(t_1)=\int_{t_1}^{t_2} J''(s) ds = c_2 \int_{t_1}^{t_2} K(u(s))ds<0
\end{align*}
for $t_2 >t_1$. 
Thus, we have $J'(t_2)< J'(t_1)$. By the similar argument as above, it holds that $J'(t) \leq J'(t_2)<0$ for any $t > t_2$. This means there exists $t^*$ such that $J(t^*)=0$. This is a contradiction to that $u$ is a non-zero forward global solution.

Next, we will show that $J'(t) \lesssim e^{-ct}$ for $t \geq 0$. 
By Corollary \ref{cor2.5} as $\varphi(x)=x^2$ and $f(x)=u(t,x)$, we obtain
\begin{align*}
	|J'(t)|^2 \lesssim (J''(t))^2 J(t)
\end{align*}
for all existence time $t$. Since $J>0$, $J'>0$, and $J''<0$, we obtain
\begin{align}
\label{eqA.2}
	\frac{J'(t)}{\sqrt{J(t)}} \lesssim -J''(t).
\end{align}
Integrating this on $(0,t)$, we get
\begin{align*}
	\sqrt{J(t)} - \sqrt{J(0)} \lesssim -J'(t)+J'(0) \lesssim J'(0).
\end{align*}
This means that $y$ is bounded on $(0,\infty)$. Using this boundedness and \eqref{eqA.2} again, we have $J'(t) \lesssim -J''(t)$.
This implies $J'(t) \lesssim e^{-ct}$ for $t \geq 0$. 
We obtain
\begin{align*}
	0 &\leq - \int_{t}^{\infty} \mu(s)ds \lesssim -\int_{t}^{\infty} K(u(s))ds 
	\\ 
	& \approx  -\int_{t}^{\infty}J''(s)ds = -[J'(s)]_{s=t}^{s=\infty} = J'(t) \lesssim e^{-ct}.
\end{align*}
This completes the proof. 
\end{proof}

\begin{corollary}
\label{cor2.8}
We assume the assumption of Proposition \ref{prop2.13} (2) and that the solution is global in positive time, then $u$ blows up in negative time. 
\end{corollary}

\begin{proof}
Suppose that $u$ is global in negative time. Set $v(t,x)=\overline{u(-t,x)}$. Then, $v$ is a solution of \eqref{NLS} satisfying the above assumption. Thus, it holds that
\begin{align*}
	\im \int_{\mathbb{R}} x \partial_x v(t,x) \overline{v(t,x)} dx >0
\end{align*}
for all $t$. We get
\begin{align*}
	0<\im \int_{\mathbb{R}} x \partial_x v(-t,x) \overline{v(-t,x)} dx
	&=\im \int_{\mathbb{R}} x \overline{\partial_x u(t,x)} u(t,x) dx
	\\
	&=-\im \int_{\mathbb{R}} x \partial_x u(t,x) \overline{u(t,x)} dx
	<0
\end{align*}
This is a contradiction. 
\end{proof}

\begin{proof}[Proof of Proposition \ref{prop2.13} (2)]
Suppose that $u$ is global in positive time direction. Then, by Corollary \ref{cor2.8}, the solution blows up in negative time. 
By Proposition \ref{prop2.7}, we have $\liminf_{t \to \infty} \mu(t)=0$. Thus, there exists a sequence $\{t_n\}$ such that $t_n \to \infty$ and $\mu(t_n) \to 0$ as $n \to \infty$.
We will prove that $\mu(t) \to 0$ as $t \to \infty$.  
If not, there exists $\varepsilon_1 \in (0,\mu_0)$ and $\{t_n'\}$ such that $-\mu(t_n') >\varepsilon_1$. We can take a sequence $\{t_n''\}$ such that 
\begin{align*}
	t_n < t_n'', \quad -\mu(t_n'')= \varepsilon_1, 
	\quad -\mu(t) < \varepsilon_1 \text{ for all } t \in [t_n,t_n'').
\end{align*}
On the interval $[t_n,t_n'']$, the parameter $\rho$ is well defined. 
By the estimate of the modulation parameter Lemma \ref{lem4.10}, we have
\begin{align*}
	|\rho (t_n'') -\rho(t_n)| \leq \int_{t_n}^{t_n''} |\rho'(t)|dt \lesssim e^{-ct_n} \to 0
\end{align*}
as $n \to \infty$. By the definition of $t_n$, we have $|\rho(t_n)|\sim |\mu(t_n)| \to 0$. However, we have $|\rho(t_n'')| \sim |\mu(t_n'')| =\varepsilon_1 > 0$ by the estimate of the modulation parameter (see Lemma \ref{lem4.5} and Corollary \ref{cor4.8}) and the definition of $t_n''$. This is a contradiction. This means that $\mu(t) \to 0$ as $t \to \infty$. 

Therefore, it follows from the estimate of the modulation parameter Lemma \ref{lem4.10} that
\begin{align*}
	|y(t_2) - y(t_1)|=\int_{t_1}^{t_2}  |y'(t)|dt \lesssim \int_{t_1}^{t_2} |\mu(t)|dt \lesssim e^{-ct_1}
\end{align*}
for large $t_2>t_1$. This implies that $y(t)$ converges to $y_{\infty} \in \mathbb{R}$ as $t \to \infty$. However, this means that
\begin{align*}
	e^{-2y(t)} \to e^{-2y_{\infty}}>0.
\end{align*}
This contradicts $e^{-2y(t)}\lesssim  |\mu(t)| \to 0$ as $t \to \infty$. As a consequence, the solution is not global in positive time direction. 
\end{proof}

\appendix
\section{Cut of odd functions}

\begin{lemma}
Let $f \in H_{\odd}^{1}(\mathbb{R})$. Then $F^{+}(x):=\1_{(0,\infty)}(x)f(x)$ also belongs to $ H^{1}(\mathbb{R})$. Moreover, we have
\begin{align*}
	\frac{d}{dx} F^{+}(x) = \1_{(0,\infty)}(x) \frac{d}{dx}f(x).
\end{align*}
\end{lemma}

\begin{proof}
We have
\begin{align*}
	\|F^{+}\|_{L^2(\mathbb{R})} 
	=\|f\|_{L^2(0,\infty)}
	\leq \|f\|_{L^2(\mathbb{R})} < \infty
\end{align*}
The weak derivative of $F^{+}$ is
\begin{align*}
	\frac{d}{dx}F^{+}(x) = \1_{(0,\infty)}(x)\frac{df}{dx}(x)
\end{align*}
Indeed, for any $\varphi \in C_{0}^{\infty}(\mathbb{R})$, we have
\begin{align*}
	\int_{\mathbb{R}} F^{+}(x) \frac{d}{dx} \varphi(x)dx
	&=\int_{0}^{\infty} f(x) \frac{d}{dx} \varphi(x)dx
	\\
	&=- \int_{0}^{\infty} \frac{d}{dx}f(x) \varphi(x)dx
	\\
	&=- \int_{\mathbb{R}}\1_{(0,\infty)} \frac{d}{dx}f(x) \varphi(x)dx
\end{align*}
by the integration by parts and $f(0)=0$. 
Moreover, 
\begin{align*}
	\|\frac{d}{dx}F^{+}\|_{L^2(\mathbb{R})}
	=\|\frac{d}{dx}f\|_{L^2(0,\infty)}
	\leq \|\frac{d}{dx}f\|_{L^2(\mathbb{R})}<\infty.
\end{align*}
Thus, we get $F^{+} \in H^1(\mathbb{R})$. 
\end{proof}

\section*{Acknowledgement}

The second author deeply appreciates the support by JSPS Overseas Research Fellowship. He also thanks OU and UBC for giving him the chance to study abroad. 
Research of the first author is partially supported
by an NSERC Discovery Grant.


\end{document}